\documentclass[a4paper]{amsart}
\usepackage{amssymb} 
\usepackage[T1]{fontenc}
\usepackage[latin1]{inputenc}
\usepackage{amsfonts}
\usepackage{amsxtra}
\usepackage{ae}
\usepackage[all]{xy}
\usepackage{enumerate}
\usepackage{url}
\usepackage{verbatim}
\usepackage{pgf,tikz}
\usepackage{tikz-cd}
\usepackage{stmaryrd}

\usetikzlibrary{arrows,patterns}

\include{diagram}

\newcommand*{\ket}{\rangle}
\newcommand*{\bra}{\langle}
\newcommand*{\A}{\mathcal{A}}

\newcommand*{\C}{\mathcal{C}}

\newcommand*{\G}{\mathcal{G}}
\renewcommand*{\H}{\mathcal{H}}

\renewcommand*{\O}{\mathcal{O}}

\newcommand*{\W}{\mathcal{W}}

\newcommand*{\Poly}{\mathcal{O}}
\newcommand*{\FC}{\mathbb{F}C}
\newcommand*{\FO}{\mathbb{F}\mathcal{O}}
\newcommand*{\Rep}{\mathsf{Rep}}

\DeclareMathOperator{\ad}{ad}

\DeclareMathOperator{\End}{End}
\DeclareMathOperator{\Aut}{Aut}

\DeclareMathOperator{\tr}{tr}
\DeclareMathOperator{\Tr}{Tr}
\DeclareMathOperator{\id}{id}

\newenvironment{bnum}
{\begin{list}{}
    {\setlength{\labelwidth}{15pt}
     \setlength{\leftmargin}{\labelwidth}
    }
}
{\end{list}}

\numberwithin{equation}{section}
\theoremstyle{change}
\newtheorem{theorem}{Theorem}[section]
\newtheorem{prop}[theorem]{Proposition}
\newtheorem{lemma}[theorem]{Lemma}

\theoremstyle{definition}
\newtheorem{definition}[theorem]{Definition}
\newtheorem{example}[theorem]{Example}
\newtheorem{remark}[theorem]{Remark}

\begin{document}

\title{Quantum Cuntz-Krieger algebras}

\author{Michael Brannan}
\address{Michael Brannan \\ Department of Mathematics\\ 
Texas A\&M University \\
College Station, TX 77840 \\
USA} 
\email{mbrannan@tamu.edu}

\author{Kari Eifler}
\address{Kari Eifler \\ Department of Mathematics\\ 
Texas A\&M University \\
College Station, TX 77840\\
USA} 
\email{keifler@tamu.edu}

\author{Christian Voigt}
\address{Christian Voigt \\ School of Mathematics and Statistics \\
University of Glasgow \\
University Place \\
Glasgow G12 8QQ \\
United Kingdom}
\email{christian.voigt@glasgow.ac.uk}

\author{Moritz Weber} 
\address{Moritz Weber \\ Saarland University \\
Faculty of Mathematics \\
Postbox 151150 \\
D-66041 Saarbr\"ucken \\
Germany}
\email{weber@math.uni-sb.de}

\begin{abstract}
Motivated by the theory of Cuntz-Krieger algebras we define and study $ C^\ast $-algebras associated to directed quantum graphs. For classical graphs 
the $ C^\ast $-algebras obtained this way can be viewed as free analogues of Cuntz-Krieger algebras, and need not be nuclear. 

We study two particular classes of quantum graphs in detail, namely the trivial and the complete quantum graphs. For the trivial quantum graph on a single 
matrix block, we show that the associated quantum Cuntz-Krieger algebra is neither unital, nuclear nor simple, and does not depend on the size of the matrix 
block up to $ KK $-equivalence. In the case of the complete quantum graphs we use quantum symmetries to show that, in certain cases, the corresponding 
quantum Cuntz-Krieger algebras are isomorphic to Cuntz algebras. These isomorphisms, which seem far from obvious from the definitions, imply in particular 
that these $ C^\ast $-algebras are all pairwise non-isomorphic for complete quantum graphs of different dimensions, even on the level of $ KK $-theory. 
We explain how the notion of unitary error basis from quantum information theory can help to elucidate the situation. 

We also discuss quantum symmetries of quantum Cuntz-Krieger algebras in general. 
\end{abstract} 


\maketitle

\section{Introduction} 

Cuntz-Krieger algebras were introduced in \cite{CKalgebras}, generalizing the Cuntz algebras in \cite{Cuntzalgebras}. 
These algebras have intimate connections with symbolic dynamics, and have been studied intensively in the framework of graph $ C^\ast $-algebras over the 
past decades, thus providing a rich supply of interesting examples \cite{Raeburngraph}. The structure of graph $ C^\ast $-algebras is understood to an impressive 
level of detail, and many properties can be interpreted geometrically in terms of the underlying graphs. Motivated by this success, the original constructions 
and results have been generalized in several directions, including higher rank graphs \cite{KPhigherrankgraphcstar}, Exel-Laca algebras \cite{ExelLacaalgebras} 
and ultragraph algebras \cite{Tomfordeunified}, among others. 

The aim of the present paper is to study a generalization of Cuntz-Krieger algebras of a quite different flavor, based on the concept of a quantum graph. The 
latter notion goes back to work of Erdos-Katavolos-Shulman \cite{EKSrankonesubspaces} and Weaver \cite{Weaverquantumrelations}, and was subsequently developed 
further by Duan-Severini-Winter \cite{DSWnoncommutativegraphs} and Musto-Reutter-Verdon \cite{MRVmorita}. Quantum graphs play an intriguing role in the study 
of the graph isomorphism game in quantum information via their connections with quantum symmetries of graphs \cite{BCEHPSWbigalois}. Moreover, based on the 
use of quantum symmetries, fascinating results on the graph theoretic interpretation of quantum isomorphisms between finite graphs were recently obtained 
by Man\v{c}inska-Roberson \cite{MRplanar}. 

Our main idea is to replace the matrix $ A $ in the definition of the Cuntz-Krieger algebra $ \O_A $ by the quantum adjacency matrix of a directed quantum graph. 
Roughly speaking, this means that the standard generators in a Cuntz-Krieger algebra are replaced by matrix-valued valued partial isometries, with matrix sizes 
determined by the quantum graph, and the Cuntz-Krieger relations are expressed using the quantum adjacency matrix of the quantum graph,  
in analogy to the scalar case. 

An important difference to the classical situation is that the matrix partial isometries are not required to have mutually orthogonal ranges, 
as this condition does not generalize to matrices in a natural way. Therefore, the quantum Cuntz-Krieger algebra of a classical graph is typically not 
isomorphic to an ordinary Cuntz-Krieger algebra, and will often neither be unital nor nuclear. However, we show that free Cuntz-Krieger algebras, or equivalently, 
quantum Cuntz-Krieger algebras associated with classical graphs, are always $ KK $-equivalent to Cuntz-Krieger algebras. 

Our main results concern the quantum Cuntz-Krieger algebras associated with two basic examples of quantum graphs, namely the trivial and complete quantum graphs 
associated to an arbitrary measured finite-dimensional $ C^\ast $-algebra $ (B,\psi) $. The first example we consider in detail, namely the quantum Cuntz-Krieger
algebra of the trivial quantum graph $ TM_N $ on the full matrix algebra $ M_N(\mathbb{C}) $, can be viewed as a non-unital version of Brown's universal algebra
generated by the entries of a unitary $ N \times N $-matrix \cite{Brownext}. For $ N > 1 $, the quantum Cuntz-Krieger algebra of $ TM_N $ is neither unital, 
nuclear nor simple, but it is always $ KK $-equivalent to $ C(S^1) $. We exhibit a description of matrices over this algebra in terms of a free product. The 
second example, namely the quantum Cuntz-Krieger algebra associated to the complete quantum graph $ K(B,\psi) $, is even more intriguing. We show
that this $ C^\ast $-algebra always admits a canonical quotient map onto the Cuntz algebra $ \O_n $, where $ n = \dim(B) $. Moreover, for certain 
quantum complete graphs we are able to show that this map is an isomorphism. This fact, which seems far from obvious from the defining relations, is 
proved using monoidal equivalence of the quantum automorphism groups of the underlying quantum graphs. In particular, our results show that for $ N > 1 $ 
the quantum Cuntz-Krieger algebras of the complete quantum graphs $ K(M_N(\mathbb{C}), \tr) $ are unital, nuclear, simple, and pairwise non-isomorphic, 
even on the level of $ KK $-theory.   

We also discuss how quantum symmetries of directed quantum graphs induce quantum symmetries of their associated quantum Cuntz-Krieger algebras in general. 
This is particularly interesting when one tries to relate quantum Cuntz-Krieger algebras associated to graphs which are quantum isomorphic, as in our 
analysis of the examples mentioned above. In particular, we indicate how the notion of a unitary error basis \cite{Wernerteleportation}, 
which is well-known in quantum information theory, can be used to find finite-dimensional quantum isomorphisms, which in turn induce crossed product relations 
between quantum Cuntz-Krieger algebras. In a sense, the existence of quantum symmetries can be viewed as a substitute for the gauge action which features 
prominently in the study of ordinary Cuntz-Krieger algebras. While there exists a gauge action in the quantum case as well, it seems to be of limited use 
for understanding the structure of quantum Cuntz-Krieger algebras in general. 

Let us briefly explain how the paper is organized. In section \ref{secck} we collect some background material on graphs and their 
associated $ C^\ast $-algebras, and introduce free graph $ C^\ast $ algebras and free Cuntz-Krieger algebras. We show that these algebras are $ KK $-equivalent 
to ordinary graph $ C^\ast $-algebras and Cuntz-Krieger algebras, respectively. After reviewing some facts about finite quantum spaces, 
that is, measured finite-dimensional $ C^\ast $-algebras, we define directed quantum graphs in section \ref{secqck}. We then introduce our main object of 
study, namely quantum Cuntz-Krieger algebras. In section \ref{secexamples} we discuss some examples of quantum graphs and their associated $ C^* $-algebras. 
We show that the quantum Cuntz-Krieger algebras associated with classical graphs lead precisely to free Cuntz-Krieger algebras, and look at several
concrete examples of quantum graphs. We also discuss two natural operations on directed quantum graphs, obtained by taking direct sums and tensor products 
of their underlying $ C^\ast $-algebras, respectively. Section \ref{secamplification} is concerned with a general procedure to assign quantum graphs to 
classical graphs, essentially by replacing all vertices with matrix blocks of a fixed size. We analyze the structure of the resulting quantum Cuntz-Krieger 
algebras, and show that they are always $ KK $-equivalent to their classical counterparts. This allows one in particular 
to determine the $ K $-theory of the quantum Cuntz-Krieger algebra of the trivial quantum graph on a matrix algebra mentioned above. 
In section \ref{secquantumsymmetry} we explain how quantum symmetries of quantum graphs induce actions on quantum Cuntz-Krieger algebras. 
We also discuss the canonical gauge action, in analogy to the classical situation. 
The construction of quantum symmetries works in fact at the level of linking algebras associated with arbitrary quantum isomorphisms of quantum graphs. 
This is used together with the some unitary error basis constructions in section \ref{secunitaryerror} to study the structure of 
the quantum Cuntz-Krieger algebras of the trivial and complete quantum graphs associated to a full matrix algebra equipped with its standard trace. In the final 
section \ref{secquantumcomplete} we gather the required results from the preceding sections to furnish a proof of our main theorem for 
quantum Cuntz-Krieger algebras of complete quantum graphs.

Let us conclude with some remarks on notation. The closed linear span of a subset $ X $ of a Banach space is denoted by $ [X] $. If $ F $ is a finite set 
and $ A $ a $ C^\ast $-algebra we shall write $ M_F(A) $ for the $ C^\ast $-algebra of matrices indexed by elements from $ F $ with entries in $ A $. 
We write $ \otimes $ both for algebraic tensor products and for the minimal tensor product of $ C^\ast $-algebras. For operators 
on multiple tensor products we use the leg numbering notation. 

\subsection*{Acknowledgements} The authors are indebted to Li Gao for fruitful discussions on unitary error bases and their connections with quantum isomorphisms. 
MB and KE were partially supported by NSF Grant DMS-2000331. CV and MW were partially supported by SFB-TRR 195 ``Symbolic Tools in Mathematics and their 
Application'' at Saarland University. Parts of this project were completed while the authors participated in the March 2019 Thematic Program 
``New Developments in Free Probability and its Applications'' at CRM (Montreal) and the October 2019 Mini-Workshop ``Operator Algebraic Quantum Groups'' 
at Mathematisches Forschungsinstitut Oberwolfach. The authors gratefully acknowledge the support and productive research environments provided by these institutes.

\section{Cuntz-Krieger algebras} \label{secck}

In this section we review the definition of Cuntz-Krieger algebras and graph $ C^\ast $-algebras \cite{CKalgebras}, \cite{EnomotoWatatanigraph}, \cite{KPRck}, 
\cite{Raeburngraph}, and introduce a free variant of these algebras. Our conventions for graphs and graph $ C^\ast $-algebras will follow \cite{KPRck}.  

\subsection{Graphs} 

A directed graph $ E = (E^0, E^1) $ consists of a set $ E^0 $ of vertices and a set $ E^1 $ of edges, 
together with source and range maps $ s, r: E^1 \rightarrow E^0 $. A vertex $ v \in E^0 $ is called a sink iff $ s^{-1}(v) $ is empty, and a source iff $ r^{-1}(v) $ 
is empty. That is, a sink is a vertex which emits no edges, and a source is a vertex which receives no edges. A self-loop is an edge with $ s(e) = r(e) $. 
The graph $ E $ is called simple if the map $ E^1 \rightarrow E^0 \times E^0, e \mapsto (s(e), r(e)) $ is injective. 

The line graph $ LE $ of $ E $ is the directed graph with vertex set $ EL^0 = E $, edge set 
$$ 
EL^1 = \{(e,f) \mid r(e) = s(f) \} \subset E \times E,  
$$
and the source and range maps $ s, r: LE^1 \rightarrow LE^0 $ given by projection to the first and second coordinates, respectively. 
By construction, the line graph $ LE $ is simple.  

The adjacency matrix of $ E = (E^0, E^1) $ is the $ E^0 \times E^0 $-matrix 
$$
B_E(v,w) = |\{e \in E^1 \mid s(e) = v, r(e) = w \}|, 
$$
and the edge matrix of $ E $ is the $ E^1 \times E^1 $-matrix with entries 
$$
A_E(e,f) = 
\begin{cases} 
1 & r(e) = s(f) \\
0 & \text{else}. 
\end{cases}
$$
Note that the edge matrix $ A_E $ of $ E $ equals the adjacency matrix $ B_{LE} $ of $ LE $. 

We will only be interested in finite directed graphs in the sequel, that is, directed graphs $ E = (E^0,E^1) $ such that both $ E^0 $ and $ E^1 $ are 
finite sets. This requirement can be substantially relaxed \cite{KPRck}.

\subsection{Graph $ C^\ast $-algebras and Cuntz-Krieger algebras} 

We recall the definition of the graph $ C^\ast $-algebra of a finite directed graph $ E = (E^0, E^1) $. 

A Cuntz-Krieger $ E $-family in a $ C^\ast $-algebra $ D $ consists of mutually orthogonal projections $ p_v \in D $ for all $ v \in E^0 $ together with partial 
isometries $ s_e \in D $ for all $ e \in E^1 $ such that 
\begin{bnum}
\item[a)] $ s_e^* s_e = p_{r(e)} $ for all edges $ e \in E^1 $ 
\item[b)] $ p_v = \sum_{s(e) = v} s_e s_e^* $ whenever $ v \in E^0 $ is not a sink. 
\end{bnum}
The graph $ C^\ast $-algebra of $ E $ can then be defined as follows. 
 
\begin{definition} \label{defgraphcstar}
Let $ E = (E^0, E^1) $ be a finite directed graph. The graph $ C^\ast $-algebra $ C^\ast(E) $ is the universal $ C^\ast $-algebra generated by a 
Cuntz-Krieger $ E $-family. We write $ P_v $ and $ S_e $ for the corresponding projections and partial isometries in $ C^\ast(E) $, associated 
with $ v \in E^0 $ and $ e \in E^1 $, respectively.  
\end{definition} 

That is, given any Cuntz-Krieger $ E $-family in a $ C^\ast $-algebra $ D $, with projections $ p_v $ for $ v \in E^0 $ and partial isometries $ s_e $ 
for $ e \in E^1 $, there exists a unique $ \ast $-homomorphism $ \phi: C^\ast(E) \rightarrow D $ such that $ \phi(P_v) = p_v $ and $ \phi(S_e) = s_e $. 

Next we recall the definition of Cuntz-Krieger algebras \cite{CKalgebras}. If $ A \in M_N(\mathbb{Z}) $ is a matrix with entries $ A(i,j) \in \{0,1 \} $
then a Cuntz-Krieger $ A $-family in a $ C^\ast $-algebra $ D $ consists of partial isometries $ s_1, \dots, s_N \in D $ 
with mutually orthogonal ranges such that the Cuntz-Krieger relations 
$$
s_i^* s_i = \sum_{j = 1}^N A(i,j) s_j s_j^*
$$
hold for all $ 1 \leq i \leq N $. 

\begin{definition} \label{defck}
Let $ A \in M_N(\mathbb{Z}) $ be a matrix with entries $ A(i,j) \in \{0,1 \} $. The Cuntz-Krieger algebra $ \O_A $ is the $ C^\ast $-algebra 
generated by a universal Cuntz-Krieger $ A $-family, that is, it is the universal $ C^\ast $-algebra generated by partial isometries $ S_1, \dots, S_N $ 
with mutually orthogonal ranges, satisfying 
$$
S_i^* S_i = \sum_{j = 1}^N A(i,j) S_j S_j^*
$$ 
for all $ 1 \leq i \leq N $.
\end{definition} 

In contrast to \cite{CKalgebras}, we do not make any further assumptions on the matrix $ A $ in Definition \ref{defck} in the sequel, except that it should 
have entries in $ \{0,1\} $. Accordingly, the algebras $ \O_A $ may sometimes be rather degenerate or even trivial, as for instance if $ A = 0 $. 
However, we have adopted this setting for the sake of consistency with our definitions in the quantum case further below. 

If $ E $ is a graph with no sinks and no sources then the graph $ C^\ast $-algebra $ C^\ast(E) $ can be canonically identified with the 
Cuntz-Krieger algebra associated with the edge matrix $ A_E $ of $ E $. 
In particular, the projections in $ C^\ast(E) $ associated to vertices of $ E $ need not be mentioned explicitly in this case. 

We note that the graph $ C^\ast $-algebra of a graph $ E $ with no sinks and no sources is completely determined by the line graph $ LE $ of $ E $, 
keeping in mind that the edge matrix $ A_E $ equals the adjacency matrix $ B_{LE} $, see also \cite{MRSrepresentationsck}. Viewing $ C^\ast(E) $ 
as being associated with the line graph of $ E $ motivates our generalizations further below, where we will replace the matrix $ A $ in Definition \ref{defck} 
with the quantum adjacency matrix of a directed quantum graph.  

\begin{remark}
It is known that all graph $ C^\ast $-algebras of finite directed graphs without sinks are isomorphic to Cuntz-Krieger algebras \cite{ARcorners}. 
\end{remark}

\subsection{Free graph $ C^\ast $-algebras and free Cuntz-Krieger algebras} 

Borrowing terminology from \cite{BSliberation}, we shall now consider ``liberated'' analogues of graph $ C^\ast $-algebras and Cuntz-Krieger algebras. 

In the case of graphs, the input for this construction is a finite directed graph $ E = (E^0, E^1) $ as above. By a free Cuntz-Krieger $ E $-family 
in a $ C^\ast $-algebra $ D $ we shall mean a collection of projections $ p_v \in D $ for all $ v \in E^0 $ together with partial isometries $ s_e \in D $ 
for all $ e \in E^1 $ such that 
\begin{bnum}
\item[a)] $ s_e^* s_e = p_{r(e)} $ for all edges $ e \in E^1 $ 
\item[b)] $ p_v = \sum_{s(e) = v} s_e s_e^* $ whenever $ v \in E^0 $ is not a sink. 
\end{bnum}
That is, the only difference to an ordinary Cuntz-Krieger $ E $-family is that the projections $ p_v $ are no longer required to be mutually orthogonal. 

Stipulating that the $ p_v $ are mutually orthogonal is equivalent to saying that the $ C^\ast $-algebra generated by the projections $ p_v $ 
is commutative. In the same way as in the liberation of matrix groups \cite{BSliberation}, removing commutation relations of this type leads to the following 
free version of the notion of a graph $ C^\ast $-algebra. 

\begin{definition} \label{defliberatedgraphcstar}
Let $ E = (E^0, E^1) $ be a finite directed graph. The free graph $ C^\ast $-algebra $ \FC^\ast(E) $ is the universal $ C^\ast $-algebra generated by a 
free Cuntz-Krieger $ E $-family. We write $ P_v $ and $ S_e $ for the corresponding projections and partial isometries in $ \FC^\ast(E) $, associated 
with $ v \in E^0 $ and $ e \in E^1 $, respectively.  
\end{definition} 

Of course, a similar definition can be made in the Cuntz-Krieger case as well. For the sake of definiteness, let us say that 
a free Cuntz-Krieger $ A $-family in a $ C^\ast $-algebra $ D $, associated with a matrix $ A \in M_N(\mathbb{Z}) $ with entries $ A(i,j) \in \{0,1 \} $, 
consists of partial isometries $ s_1, \dots, s_N \in D $ such that the Cuntz-Krieger relations 
$$
s_i^* s_i = \sum_{j = 1}^N A(i,j) s_j s_j^*
$$
hold for all $ 1 \leq i \leq N $. 
 
\begin{definition} \label{defliberatedck}
Let $ A \in M_N(\mathbb{Z}) $ be a matrix with entries $ A(i,j) \in \{0,1 \} $. The free Cuntz-Krieger algebra $ \FO_A $ is the universal $ C^* $-algebra 
generated by partial isometries $ S_1, \dots, S_N $, satisfying the relations 
$$
S_i^* S_i = \sum_{j = 1}^N A(i,j) S_j S_j^*
$$ 
for all $ i $. 
\end{definition} 

We note that free graph $ C^\ast $-algebras and free Cuntz-Krieger algebras always exist, keeping in mind that the norms of all generators are uniformly 
bounded in any representation of the universal $ \ast $-algebra generated by a free Cuntz-Krieger family. 
Let us also remark that the free graph $ C^* $-algebra of a finite directed graph $ E $ with no sinks and no sources agrees with the free Cuntz-Krieger algebra 
associated with the edge matrix $ A_E $. 

For any finite directed graph $ E $ and any matrix $ A $ as above there are canonical surjective $ \ast $-homomorphisms $ \pi: \FC^\ast(E) \rightarrow C^\ast(E) $
and $ \pi: \FO_A \rightarrow \O_A $, respectively, obtained directly from the universal property. These maps are not isomorphisms in general. 

For instance, if $ E $ is the graph with two vertices and no edges then $ C^\ast(E) = \mathbb{C} \oplus \mathbb{C} $, 
whereas $ \FC^\ast(E) = \mathbb{C} * \mathbb{C} $ is the non-unital free product of two copies of $ \mathbb{C} $. 
However, we note that if $ E $ is the graph with a single vertex and $ N $ self-loops then the canonical projection induces an 
isomorphism $ \FC^\ast(E) \cong C^\ast(E) $, identifying the free graph $ C^\ast $-algebra with the Cuntz algebra $ \O_N $. 

Let us now elaborate on the relation between $ \FC^\ast(E) $ and $ C^\ast(E) $ for an arbitrary finite directed graph $ E $, 
and similarly on the relation between $ \FO_A $ and $ \O_A $. 

\begin{theorem} \label{graphversusfreegraph}
Let $ E $ be a finite directed graph. Then the canonical projection map $ \FC^\ast(E) \rightarrow C^\ast(E) $ is a $ KK $-equivalence. 
Similarly, if $ A \in M_N(\mathbb{Z}) $ is a matrix with entries $ A(i,j) \in \{0,1 \} $ then the canonical projection $ \FO_A \rightarrow \O_A $
is a $ KK $-equivalence. 
\end{theorem} 

\begin{proof} 
The proof is analogous for graph algebras and Cuntz-Krieger algebras, therefore we shall restrict attention to the case of graph algebras. 

Adapting a well-known argument from \cite{Cuntzfreeproduct}, 
we will show more generally that $ C^\ast(E) $ and $ \FC^\ast(E) $ cannot be distinguished by any homotopy invariant functor on the category 
of $ C^\ast $-algebras which is stable under tensoring with finite matrix algebras. 

Firstly, we claim that there exists a $ \ast $-homomorphism $ \phi: C^\ast(E) \rightarrow M_{E^0}(\FC^\ast(E)) $ satisfying 
\begin{align*} 
\phi(P_v)_{xy} &= \delta_{x,v} \delta_{y,v} P_v, \\
\phi(S_e)_{xy} &= \delta_{x, s(e)} \delta_{y, r(e)} S_e 
\end{align*}
for $ v \in E^0 $ and $ e \in E^1 $. For this it suffices to show that the elements $ \phi(P_v), \phi(S_e) $ in $ M_{E^0}(\FC^\ast(E)) $ given by the above 
formulas define a Cuntz-Krieger $ E $-family. Clearly, the elements $ P_v $ are mutually orthogonal projections, 
and the elements $ \phi(S_e) $ are partial isometries such that 
$$
(\phi(S_e)^* \phi(S_e))_{xy} = \delta_{x, r(e)} \delta_{y, r(e)} S_e^* S_e = \delta_{x, r(e)} \delta_{y, r(e)} P_{r(e)} = \phi(P_{r(e)})_{xy} 
$$
and 
\begin{align*}
\phi(P_v)_{xy} = \delta_{x, v} \delta_{y, v} P_v 
&= \delta_{x, v} \delta_{y, v} \sum_{s(f) = v} S_f S_f^* \\ 
&= \sum_{s(f) = v} \delta_{x, s(f)} \delta_{y, s(f)}S_f S_f^* \\ 
&= \sum_{s(f) = v} \sum_{z \in E^0} \phi(S_f)_{xz} (\phi(S_f)_{yz})^* \\
&= \sum_{s(f) = v} (\phi(S_f) \phi(S_f)^*)_{xy}
\end{align*}
if $ v \in E^0 $ is not a sink, as required. 

Recall that we write $ \pi: \FC^\ast(E) \rightarrow C^*(E) $ for the canonical projection. 
Fixing a vertex $ w \in E^0 $, we claim that $ M_{E^0}(\pi) \circ \phi $ is homotopic to the embedding $ \iota $ of $ C^\ast(E) $ 
into the corner of $ M_{E^0}(C^\ast(E)) $ corresponding to $ w $. 
For this we consider the $ \ast $-homomorphisms $ \mu_t: C^\ast(E) \rightarrow M_{E^0}(C^\ast(E)) $ for $ t \in [0,1] $ given by 
$$
\mu_t(P_v) = u^v_t \iota(P_v) (u^v_t)^*, \qquad \mu_t(S_e) = u^{s(e)}_t \iota(S_e) (u^{r(e)}_t)^*, 
$$
where $ u^x_t $ for $ x \in E^0 $ with $ x \neq w $ is the rotation matrix 
$$
u_t = 
\begin{pmatrix} 
\cos(2 \pi t) & \sin(2 \pi t) \\
-\sin(2 \pi t) & \cos(2 \pi t) 
\end{pmatrix} 
$$
placed in the block corresponding to the indices $ w $ and $ x $, and $ u^x_t = \id $ for $ x = w $. In a similar way as above one checks that $ \mu_t $ 
preserves the relations for $ C^\ast(E) $. Indeed, the elements $ \mu_t(P_v) $ are mutually orthogonal projections 
since $ P_v, P_w $ for $ v \neq w $ are orthogonal in $ C^\ast(E) $ and the unitaries $ u^x_t $ have scalar entries. 
Moreover, for $ t \in [0,1] $ the elements $ \mu_t(S_e) $ are partial isometries such that 
$$
\mu_t(S_e)^* \mu_t(S_e) = u^{r(e)}_t \iota(S_e^* S_e) (u^{r(e)}_t)^* = u^{r(e)}_t \iota(P_{r(e)}) (u^{r(e)}_t)^* = \mu_t(P_{r(e)}),  
$$
and 
\begin{align*}
\mu_t(P_v) &= u^v_t \iota(P_v) (u^v_t)^* 
= \sum_{s(f) = v} u^{s(f)}_t \iota(S_f S_f^*) (u^{s(f)}_t)^* 
= \sum_{s(f) = v} \mu_t(S_f) \mu_t(S_f)^*
\end{align*}
if $ v $ is not a sink. By construction we have $ \mu_0 = \iota $ and $ \mu_1 = M_{E^0}(\pi) \circ \phi $. 

The composition $ \phi \circ \pi $ looks the same as $ M_{E^0}(\pi) \circ \phi $ on generators, and a similar homotopy shows that $ \phi \circ \pi $
is homotopic to the embedding $ \FC^\ast(E) \rightarrow M_{E^0}(\FC^\ast(E)) $ associated with a fixed vertex $ w $. This finishes the proof. 
\end{proof}

\section{Quantum Cuntz-Krieger algebras} \label{secqck} 

In this section we define our quantum analogue of Cuntz-Krieger algebras. Since the input for this construction is the quantum adjacency 
matrix of a directed quantum graph, we shall first review the concept of a quantum graph. 

\subsection{Quantum graphs} \label{parqgraph}

The notion of a quantum graph has been considered with some variations by a number of authors, see \cite{EKSrankonesubspaces}, \cite{Weaverquantumrelations}, 
\cite{DSWnoncommutativegraphs}, \cite{MRVmorita}, \cite{BCEHPSWbigalois}. We will follow the approach in \cite{MRVmorita}, \cite{BCEHPSWbigalois}, and adapt 
it to the setting of directed graphs. 

Assume that $ B $ is a finite dimensional $ C^\ast $-algebra $ B $ and let $ \psi: B \rightarrow \mathbb{C} $ be a faithful state. We 
denote by $ L^2(B) = L^2(B, \psi) $ the Hilbert space obtained by equipping $ B $ with the inner product $ \bra x, y \ket = \psi(x^* y) $. 
Moreover let us write $ m: B \otimes B \rightarrow B $ for the multiplication map of $ B $ and $ m^* $ for its adjoint, noting that $ m $ can be 
viewed as a linear operator $ L^2(B) \otimes L^2(B) \rightarrow L^2(B) $. 

If $ B = C(X) $ is the algebra of functions on a finite set $ X $ then states on $ B $ correspond to probability measures on $ X $. 
The most natural choice is to take for $ \psi $ the state corresponding to the uniform measure in this case. 
For an arbitrary finite dimensional $ C^\ast $-algebra $ B $ we have the following well-known condition, singling out certain natural choices among all possible 
states on $ B $ in a similar way \cite{Banicafusscatalan}. 

\begin{definition} \label{deffqs}
Let $ B $ be a finite dimensional $ C^\ast $-algebra and $ \delta > 0 $. A faithful state $ \psi: B \rightarrow \mathbb{C} $ is called a $ \delta $-form 
if $ m m^* = \delta^2 \id $. 
By a finite quantum space $ (B, \psi) $ we shall mean a finite dimensional $ C^\ast $-algebra $ B $ together with 
a $ \delta $-form $ \psi: B \rightarrow \mathbb{C} $. 
\end{definition} 

If $ B $ is a finite dimensional $ C^\ast $-algebra then we have $ B \cong \bigoplus_{a = 1}^d M_{N_a}(\mathbb{C}) $ for some $ N_1, \dots, N_d $. 
A state $ \psi $ on $ B $ can be written uniquely in the form  
$$ 
\psi(x) = \sum_{a = 1}^d \Tr(Q_{(a)} x_i)  
$$
for $ x = (x_1, \dots, x_d) $, where the $ Q_{(a)} \in M_{N_a}(\mathbb{C}) $ are positive matrices satisfying $ \sum_{a = 1}^d \Tr(Q_{(a)}) = 1 $. 
Then $ \psi $ is a $ \delta $-form iff $ Q_{(a)} $ is invertible and $ \Tr(Q_{(a)}^{-1}) = \delta^2 $ for all $ a $. Here $ \Tr $ denotes the natural 
trace, given by summing all diagonal terms of a matrix. 

Note that we may assume without loss of generality that all matrices $ Q_{(a)} $ in the definition of $ \psi $ are diagonal. We shall say that $ (B, \psi) $ 
as above is in standard form in this case. 

Any finite dimensional $ C^\ast $-algebra $ B $ admits a unique tracial $ \delta $-form for a uniquely determined value of $ \delta $. Explicitly, 
this is the tracial state given by  
$$
\tr(x) = \frac{1}{\dim(B)} \sum_{a = 1}^d N_a \Tr(x_i), 
$$
and we have $ \delta^2 = \dim(B) $. Note that if $ B = C(X) $ is commutative then this corresponds to the uniform measure on $ X $, and $ \delta^2 $ is the 
cardinality of $ X $. 

For later purposes it will be useful to record an explicit formula for the adjoint of the multiplication map in a finite quantum space. 

\begin{lemma} \label{mstarcomputation}
Let $ (B, \psi) $ be a finite quantum space in standard form as described above, and consider the linear basis of $ B $ given by the adapted matrix units 
$$
f^{(a)}_{ij} = (Q_{(a)}^{-1/2})_{ii} e^{(a)}_{ij} (Q_{(a)}^{-1/2})_{jj}, 
$$
where $ e^{(a)}_{ij} $ in $ M_{N_a}(\mathbb{C}) $ are the standard matrix units. 
Then we have $ (f^{(a)}_{ij})^* = f^{(a)}_{ji} $ and 
$$
m^*(f^{(a)}_{ij}) = \sum_k f^{(a)}_{ik} \otimes f^{(a)}_{kj}
$$
for all $ a,i,j $. 
\end{lemma} 

\begin{proof} 
Since the matrices $ Q_{(a)} $ are positive we clearly have $ (f^{(a)}_{ij})^* = f^{(a)}_{ji} $. Moreover, observing 
$$ 
f^{(b)}_{rs} f^{(c)}_{pq} = \delta_{bc} (Q_{(b)}^{-1})_{sp} f^{(b)}_{rq} 
$$ 
and $ \psi(f^{(a)}_{kl}) = \delta_{kl} $, we compute 
\begin{align*} 
\bra f^{(a)}_{ij}, m(f^{(b)}_{rs} \otimes f^{(c)}_{pq}) \ket 
&= \delta_{bc} \psi(f^{(a)}_{ji} (Q_{(b)}^{-1})_{sp} f^{(b)}_{rq}) 
= \delta_{abc} (Q_{(a)}^{-1})_{sp} (Q_{(a)}^{-1})_{ir} \delta_{jq}  
\end{align*}
and 
\begin{align*}
\sum_k \bra f^{(a)}_{ik} \otimes f^{(a)}_{kj}, f^{(b)}_{rs} \otimes f^{(c)}_{pq} \ket
&= \sum_k \psi(f^{(a)}_{ki} f^{(b)}_{rs}) \psi(f^{(a)}_{jk} f^{(c)}_{pq}) \\
&= \delta_{ab} \delta_{ac} \sum_k (Q_{(a)}^{-1})_{ir} (Q_{(a)}^{-1})_{kp} \psi(f^{(a)}_{ks}) \psi(f^{(a)}_{jq}) \\
&= \delta_{abc} (Q_{(a)}^{-1})_{ir} (Q_{(a)}^{-1})_{sp} \delta_{jq}.  
\end{align*}
This yields the claim. 
\end{proof}

Let us now discuss the concept of a quantum graph. We shall be interested in directed quantum graphs in the following sense. 

\begin{definition} \label{defdirectedqgraph}
Let $ B $ be a finite dimensional $ C^\ast $-algebra and $ \psi: B \rightarrow \mathbb{C} $ a $ \delta $-form. A linear operator $ A: L^2(B) \rightarrow L^2(B) $ 
is called a quantum adjacency matrix if 
$$ 
m(A \otimes A) m^* = \delta^2 A. 
$$
A directed quantum graph $ \G = (B, \psi, A) $ is a finite quantum space $ (B, \psi) $ together with a quantum adjacency matrix. 
\end{definition} 

In order to explain Definition \ref{defdirectedqgraph} let us consider the case that $ B = C(X) $ is the quantum space associated with a finite set $ X $, 
with $ \psi $ being given by the uniform measure. A linear operator $ A: L^2(B) \rightarrow L^2(B) $ can be identified canonically with a matrix 
in $ M_X(\mathbb{C}) $. Moreover, a straightforward calculation shows that 
$$ 
\frac{1}{|X|} m(A \otimes B) m^* 
$$ 
is the Schur product of $ A, B \in M_X(\mathbb{C}) $, given by entrywise multiplication.  
Hence $ A $ is a quantum adjacency matrix iff it is an idempotent with respect to the Schur product in this case, which is equivalent to saying that $ A $ has 
entries in $ \{0,1\} $. 

According to the above discussion, every simple finite directed classical graph $ E = (E^0, E^1) $ gives rise to a directed quantum graph 
in a natural way. More precisely, if $ A_E $ denotes the adjacency matrix of $ E $ then we obtain a 
directed quantum graph structure on $ B = C(E^0) $ by taking the state $ \psi $ which corresponds to counting measure, 
and the operator $ A: L^2(B) \rightarrow L^2(B) $ given by $ A(e_i) = \sum_j A(i,j) e_j $. 
Conversely, every directed quantum graph structure on a finite dimensional commutative $ C^\ast $-algebra $ B = C(X) $ arises from 
a simple finite directed graph on the vertex set $ X $. 

For a general finite quantum space $ (B, \psi) $ it will be convenient for our considerations further below to write down the quantum adjacency 
matrix condition in terms of bases. 

\begin{lemma} \label{adjacencycomputation}
Let $ (B, \psi) $ be a finite quantum space in standard form. Then a linear operator $ A: L^2(B) \rightarrow L^2(B) $, given by 
$$
A(f^{(a)}_{ij}) = \sum_{brs} A_{ija}^{rsb} f^{(b)}_{rs}
$$ 
in terms of the adapted matrix units, is a directed quantum adjacency matrix iff
$$
\sum_{ks} (Q_{(b)}^{-1})_{ss} A_{ika}^{rsb} A_{kja}^{snb} = \delta^2 A_{ija}^{rnb}
$$
for all $ a,b,i,j,r,n $.  
\end{lemma} 

\begin{proof} 
Using Lemma \ref{mstarcomputation} we calculate 
\begin{align*}
m(A \otimes A) m^*(f^{(a)}_{ij}) &= \sum_k A(f^{(a)}_{ik}) A(f^{(a)}_{kj}) \\
&= \sum_k \sum_{brs} \sum_{cmn} A_{ika}^{rsb} f^{(b)}_{rs} A_{kja}^{mnc} f^{(c)}_{mn} \\
&= \sum_k \sum_{brsn} (Q_{(b)}^{-1})_{ss} A_{ika}^{rsb} A_{kja}^{snb} f^{(b)}_{rn}, 
\end{align*}
so that comparing coefficients yields the claim. 
\end{proof} 

We point out that there is a rich supply of directed quantum adjacency matrices and quantum graphs. Let $ B $ be a finite dimensional $ C^\ast $-algebra and 
let $ \tr $ be the unique tracial $ \delta $-form on $ B $. Every element $ P \in B \otimes B^{op} $ has a Choi-Jamio\l kowski form, that is, there exists a 
unique linear map $ A: B \rightarrow B $ such that 
$$
P = P_A = \frac{1}{\dim(B)} (1 \otimes A)m^*(1),  
$$
where $ m^*: B \rightarrow B \otimes B $ is the adjoint of multiplication with respect to $ \tr $. 
Then $ A $ is a quantum adjacency matrix with respect to $ (B, \tr) $ iff $ P $ is idempotent, that is, iff $ P^2 = P $. 

Moreover, idempotents in $ B \otimes B^{op} $ can be naturally obtained as follows. Assume that $ B \hookrightarrow B(\H) $ is unitally embedded into 
the algebra of bounded operators on some finite dimensional Hilbert space $ \H $, and let $ B' \subset B(\H) $ be the commutant of $ B $. 
Then $ B \otimes B^{op} $ identifies with the space of all completely bounded $ B' $-$ B' $-bimodule maps from $ B(\H) $ to itself. 
In particular, idempotents in $ B \otimes B^{op} $ are the same thing as direct sum decompositions $ B(\H) = S \oplus R $ of $ B' $-$ B' $-bimodules. 

Taking $ B = M_N(\mathbb{C}) $ and the identity embedding into $ B(\mathbb{C}^N) = M_N(\mathbb{C}) $ we see that there is a bijective 
correspondence between quantum graph structures on $ (M_N(\mathbb{C}), \tr) $ and vector space direct sum decompositions $ M_N(\mathbb{C}) = S \oplus R $. 

\begin{remark} 
One could work more generally with arbitrary faithful positive linear functionals $ \psi $ instead of $ \delta $-forms, by modifying the 
defining relation of a quantum adjacency matrix in Definition \ref{defdirectedqgraph} to 
$$ 
m(A \otimes A) m^* = A mm^*. 
$$
We will however restrict ourselves to $ \delta $-forms in the sequel, as this will allow us to remain closer to the classical theory in the 
commutative case. 
\end{remark} 

\begin{remark}
The definition of a quantum graph in \cite{MRVmorita}, \cite{BCEHPSWbigalois} contains further conditions on the quantum adjacency matrix. If $ B = C(X) $ 
is commutative then these conditions correspond to requiring that the matrix $ A \in M_X(\mathbb{C}) $ is symmetric and has entries $ 1 $ 
on the diagonal, respectively. That is, the quantum graphs considered in these papers are undirected and have all self-loops. Neither of these conditions 
is appropriate in connection with Cuntz-Krieger algebras. 
\end{remark}

\subsection{Quantum Cuntz-Krieger algebras} 

Let us now define the quantum Cuntz-Krieger algebra associated to a directed quantum graph. Comparing with the definition of graph $ C^\ast $-algebras, 
we note that the quantum graph used as an input in our definition may be thought of as an analogue of the line graph of a classical graph. 

If $ \G = (B, \psi, A) $ is a directed quantum graph then we shall say that a quantum Cuntz-Krieger $ \G $-family in a $ C^\ast $-algebra $ D $ is a linear 
map $ s: B \rightarrow D $ such that 
\begin{bnum} 
\item[a)] $ \mu_D(\id \otimes \mu_D)(s \otimes s^* \otimes s)(\id \otimes m^*)m^* = s $
\item[b)] $ \mu_D(s^* \otimes s)m^* = \mu_D(s \otimes s^*) m^* A $, 
\end{bnum}
where $ \mu_D: D \otimes D \rightarrow D $ is the multiplication map for $ D $ and $ s^*(b) = s(b^*)^* $ for $ b \in B $. 
We also recall that $ m^* $ denotes the adjoint of the multiplication map for $ B $ with respect to the inner product given by $ \psi $. 

\begin{definition} \label{defqck}
Let $ \G = (B, \psi, A) $ be a directed quantum graph. The quantum Cuntz-Krieger algebra $ \FO(\G) $ is the universal $ C^\ast $-algebra 
generated by a quantum Cuntz-Krieger $ \G $-family $ S: B \rightarrow \FO(\G) $. 
\end{definition} 

In other words, the quantum Cuntz-Krieger algebra $ \FO(\G) $ satisfies the following universal property. If $ D $ is a $ C^\ast $-algebra 
and $ s: B \rightarrow D $ a quantum Cuntz-Krieger $ \G $-family, then there exists a unique $ \ast $-homomorphism $ \varphi: \FO(\G) \rightarrow D $ 
such that $ \varphi(S(b)) = s(b) $ for all $ b \in B $. 

\begin{remark}
We note that Definition \ref{defqck} makes sense for a finite dimensional $ C^\ast $-algebra $ B $ together with a faithful positive linear 
functional $ \psi $ and an arbitrary linear map $ A: L^2(B) \rightarrow L^2(B) $. At this level of generality one can shift information from $ \psi $ into 
the matrix $ A $ and vice versa, without changing the resulting $ C^\ast $-algebra. 
Our definition will allow us to remain closer to the standard setup for Cuntz-Krieger algebras. 
\end{remark}

It is not difficult to check that 
the quantum Cuntz-Krieger algebra $ \FO(\G) $ always exists. This is done most easily by rewriting Definition \ref{defqck} in terms of a linear basis for 
the algebra $ B $. In the sequel we shall say that a directed quantum graph $ \G = (B, \psi, A) $ is in standard form if its underlying finite quantum space 
is, compare paragraph \ref{parqgraph}.  

\begin{prop} \label{qckconcrete}
Let $ \G = (B, \psi, A) $ be a directed quantum graph in standard form, and let 
\begin{align*} 
A(f^{(a)}_{ij}) = \sum_{brs} A_{ija}^{rsb} f^{(b)}_{rs} 
\end{align*}
be the quantum adjacency matrix written in terms of the adapted matrix units as discussed further above. 
Then the quantum Cuntz-Krieger algebra $ \FO(\G) $ identifies with the universal $ C^\ast $-algebra $ \FO_A $ with generators $ S^{(a)}_{ij} $ 
for $ 1 \leq a \leq d $ and $ 1 \leq i,j \leq N_a $, satisfying the relations 
\begin{align*}
\sum_{rs} S^{(a)}_{ir} (S^{(a)}_{sr})^* S^{(a)}_{sj} &= S^{(a)}_{ij} \\
\sum_l (S^{(a)}_{li})^* S^{(a)}_{lj} &= \sum_{brs} A^{rsb}_{ija} \sum_l S^{(b)}_{rl} (S^{(b)}_{sl})^*
\end{align*}
for all $ a,i,j $. 
\end{prop} 

\begin{proof} 
Let us first consider the elements $ S^{(a)}_{ij} = S(f^{(a)}_{ij}) $ in $ \FO(\G) $. If $ \mu $ denotes the multiplication map for $ \FO(\G) $, then 
according to Lemma \ref{mstarcomputation} we get  
\begin{align*}
\sum_{rs} S^{(a)}_{ir} (S^{(a)}_{sr})^* S^{(a)}_{sj} &= \sum_{rs} S(f^{(a)}_{ir}) S^*(f^{(a)}_{rs}) S(f^{(a)}_{sj}) \\
&= \sum_{rs} \mu(\id \otimes \mu)(S \otimes S^* \otimes S)(f^{(a)}_{ir} \otimes f^{(a)}_{rs} \otimes f^{(a)}_{sj}) \\
&= \mu(\id \otimes \mu)(S \otimes S^* \otimes S)(\id \otimes m^*)m^*(f^{(a)}_{ij}) \\
&= S(f^{(a)}_{ij}) = S^{(a)}_{ij}, 
\end{align*}
and similarly 
\begin{align*}
\sum_r (S^{(a)}_{ri})^* S^{(a)}_{rj} &= \sum_r \mu(S^* \otimes S)(f^{(a)}_{ir} \otimes f^{(a)}_{rj}) \\
&= \mu(S^* \otimes S)m^*(f^{(a)}_{ij}) \\
&= \mu(S \otimes S^*)m^* A(f^{(a)}_{ij}) \\
&= \sum_{brs} A_{ija}^{rsb} \mu(S \otimes S^*)m^* (f^{(b)}_{rs}) \\
&= \sum_{brsl} A_{ija}^{rsb} \mu(S \otimes S^*)(f^{(b)}_{rl} \otimes f^{(b)}_{ls}) \\
&= \sum_{brsl} A_{ija}^{rsb} S^{(b)}_{rl} (S^{(b)}_{sl})^*.  
\end{align*}
Hence, by the definition of $ \FO_A $, there exists a unique $ \ast $-homomorphism $ \phi: \FO_A \rightarrow \FO(\G) $ such 
that $ \phi(S^{(a)}_{ij}) = S(f^{(a)}_{ij}) $ for all $ a,i,j $. 

Conversely, let us define a linear map $ s: B \rightarrow \FO_A $ by $ s(f^{(a)}_{ij}) = S^{(a)}_{ij} $. 
Essentially the same computation as above shows that $ s $ defines a quantum Cuntz-Krieger $ \G $-family in $ \FO_A $, so that there 
exists a unique $ \ast $-homomorphism $ \psi: \FO(\G) \rightarrow \FO_A $ satisfying $ \psi(S(b)) = s(b) $ for all $ b \in B $. 

It is straightforward to check that the maps $ \phi $ and $ \psi $ are mutually inverse isomorphisms. 
\end{proof}

Using matrix notation we can rephrase the relations from Proposition \ref{qckconcrete} in a very concise way. 
More precisely, writing $ S^{(a)} \in M_{N_a}(\FO(\G)) $ for the matrix with entries $ S^{(a)}_{ij} = S(f^{(a)}_{ij}) $ 
and $ \hat{A} $ for the $ d \times d $-matrix with coefficients $ \hat{A}^b_a = A_{ija}^{rsb} $ we obtain 
\begin{align*}
S^{(a)} (S^{(a)})^* S^{(a)} &= S^{(a)} \\
(S^{(a)})^* S^{(a)} &= \sum_b \hat{A}^b_a S^{(b)} (S^{(b)})^*
\end{align*}
for all $ 1 \leq a \leq d $. The first formula says that the elements $ S^{(a)} \in M_{N_a}(\FO(\G)) $ are partial isometries. This means in particular that 
their entries are bounded in norm by $ 1 $, which implies in turn that the universal $ C^\ast $-algebras $ \FO_A $ and $ \FO(\G) $ 
always exist. The second formula can be viewed as a matrix-valued version of the classical Cuntz-Krieger relation. 

\begin{remark}
From Proposition \ref{qckconcrete} and the above remarks it may appear at first sight that $ \FO(\G) \cong \FO_A $ does not depend on 
the $ \delta $-form $ \psi $ in $ \G = (B, \psi, A) $. However, recall from Lemma \ref{adjacencycomputation} that the choice of $ \psi $ is 
reflected in the defining relations for the coefficients $ A_{ija}^{rsb} $ of the quantum adjacency matrix. 
\end{remark}

\begin{remark} 
As will be discussed in more detail at the start of the next section, the notation $ \FO_A $ used in Proposition \ref{qckconcrete} is compatible 
with our notation for free Cuntz-Krieger algebras introduced in Definition \ref{defliberatedck}.  
\end{remark}

\section{Examples} \label{secexamples}

In this section we take a look at some examples of quantum graphs and their associated quantum Cuntz-Krieger algebras in the sense of 
Definition \ref{defqck}. 

\subsection{Classical graphs} \label{parclassical}

Assume that $ E = (E^0, E^1) $ is a finite simple directed graph with $ N $ vertices. 
The directed quantum graph $ \G $ associated with $ E $ has $ B = C(E^0) = \mathbb{C}^N $ as underlying $ C^\ast $-algebra. We work with the canonical 
basis $ e_1, \dots, e_N $ of minimal projections in $ B $ and the normalized standard trace $ \tr: B \rightarrow \mathbb{C} $. That is, $ \tr(e_i) = 1/N $ 
for all $ i $, and we have $ m(e_i \otimes e_j) = \delta_{ij} e_i $ and $ m^*(e_i) = N e_i \otimes e_i $. 
If $ B_E $ denotes the adjacency matrix of $ E $ then 
$$ 
A(e_i) = \sum_{j = 1}^N B_E(i,j) e_j 
$$
determines a quantum adjacency matrix $ A: L^2(B) \rightarrow L^2(B) $. 

\begin{prop} \label{qgraphalgebraclassical}
Let $ E $ be a finite simple directed graph and let $ \G = (B, \psi, A) $ be the quantum graph corresponding to $ E $ as above. Then the 
free Cuntz-Krieger algebra associated with the adjacency matrix $ B_E $ of $ E $ is canonically isomorphic to the quantum Cuntz-Krieger algebra $ \FO(\G) $. 
\end{prop} 

\begin{proof} 
This can be viewed as a special case of Proposition \ref{qckconcrete}, but let us write down the key formulas explicitly. 
Note that $ \tr $ is a $ \delta $-form with $ \delta^2 = N $ and consider $ S_i = N S(e_i) \in \FO(\G) $.  
Then the defining relations for a free Cuntz-Krieger $ B_E $-family are obtained from 
\begin{align*}
S_i S_i^* S_i &= N^3 \mu(\id \otimes \mu)(S(e_i) \otimes S^*(e_i) \otimes S(e_i)) \\
&= N \mu(\id \otimes \mu)(S \otimes S^* \otimes S)(\id \otimes m^*)m^*(e_i) \\
&= N S(e_i) = S_i
\end{align*}
and 
\begin{align*}
S_i^* S_i &= N^2 \mu(S^* \otimes S)(e_i \otimes e_i) \\
&= N\mu(S^* \otimes S)m^*(e_i) \\
&= N \mu(S \otimes S^*)m^*(A(e_i)) \\
&= N^2 \sum_{j = 1}^N B_E(i,j) \mu(S \otimes S^*)(e_j \otimes e_j) \\
&= \sum_{j = 1}^N B_E(i,j) S_j S_j^* 
\end{align*} 
for all $ i $. This yields a $ \ast $-homomorphism $ \FO_{B_E} \rightarrow \FO(\G) $. 
Similarly, one checks that the linear map $ s: B \rightarrow \FO_{B_E} $ given by $ s(e_i) = \frac{1}{N} S_i $ is a quantum Cuntz-Krieger $ \G $-family,  
which induces a $ \ast $-homomorphism $ \FO(\G) \rightarrow \FO_{B_E} $. These maps are mutually inverse isomorphisms. 
\end{proof} 

It follows from the remarks after Definition \ref{defdirectedqgraph} that every quantum Cuntz-Krieger algebra $ \FO(\G) $ over a directed quantum 
graph $ \G = (B, \psi, A) $ with $ B $ abelian 
is a free Cuntz-Krieger algebra associated to some $ 0,1 $-matrix, and conversely, all free Cuntz-Krieger algebras arise in this way. 

Let us also point out that already quantum Cuntz-Krieger algebras associated with classical graphs as in Proposition \ref{qgraphalgebraclassical} 
may fail to be unital. This is of course in contrast to the situation for ordinary Cuntz-Krieger algebras.

\subsection{Complete quantum graphs and quantum Cuntz algebras} \label{parQKn}

Let us next consider an arbitrary finite quantum space $ (B, \psi) $ in standard form, using the same notation as after Definition \ref{deffqs}. 
Following \cite{BCEHPSWbigalois}, we can form the {\it complete quantum graph} on $ (B,\psi) $, which is the directed quantum 
graph $ K(B,\psi) = (B,\psi,A) $ with quantum adjacency matrix $ A: L^2(B) \to L^2(B) $ given by $ A(b) = \delta^2 \psi(b)1 $. 
In terms of the adapted matrix units $ f_{ij}^{(a)} \in B $ defined in Lemma \ref{mstarcomputation} we get 
$$
A(f_{ij}^{(a)}) = \delta_{ij} \delta^2 1 = \sum_{b} \sum_k \delta_{ij}\delta^2 (Q_{(a)})_{kk}f^{(b)}_{kk}. 
$$  
Therefore, relative to this basis, we have the matrix representation $ A = (A^{klb}_{ija}) $, where
\begin{align*}
A^{klb}_{ija} = \delta_{ij}\delta_{kl} \delta^2(Q_{(b)})_{kk}. 
\end{align*}
It follows from Proposition \ref{qckconcrete} and the preceding discussion that the quantum Cuntz-Krieger algebra $ \FO(K(B,\psi)) $ is the 
universal $ C^\ast $-algebra with generators $ S_{ij}^{(a)} $ for $ 1 \leq a \leq d, 1 \leq i,j \leq N_a $ and relations 
\begin{align*}
\sum_{rs} S^{(a)}_{ir} (S^{(a)}_{sr})^* S^{(a)}_{sj} &= S^{(a)}_{ij}, \\
\sum_r (S^{(a)}_{ri})^* S^{(a)}_{rj} &= \delta_{ij} \delta^2 \sum_{b} \sum_{kl} (Q_{(b)})_{kk}S^{(b)}_{kl} (S^{(b)}_{kl})^*
\end{align*}
for all $ a,i,j $.   

\begin{example} \label{completematrixqgraph}
Let us consider explicitly the special case of the complete quantum graph $ K(M_N(\mathbb{C}), \tr) $ on a full matrix algebra $ B = M_N(\mathbb{C}) $. 
The $ C^\ast $-algebra $ \FO(K(M_N(\mathbb{C}), \tr)) $ has generators $ S_{ij} $ for $ 1 \leq i,j \leq N $ satisfying the relations 
\begin{align*}
\sum_{kl} S_{ik} S^*_{lk} S_{lj} &= S_{ij} \\
\sum_r S^*_{ri} S_{rj} &= \delta_{ij} N \sum_{rs} S_{rs} S^*_{rs} 
\end{align*}
for all $ i,j $. 
\end{example} 

Note that when $ B = \mathbb{C}^d $ is abelian, Proposition \ref{qgraphalgebraclassical} implies that $ \FO(K(\mathbb{C}^d, \tr)) $ is nothing other than 
the free Cuntz-Krieger algebra associated to the complete graph $ K_d $, 
or equivalently, the free graph $ C^\ast $-algebra associated to the graph with a single vertex and $ d $ self-loops. 
Thus $ \FO(K(\mathbb{C}^d, \tr)) $ identifies with the Cuntz algebra $ \O_d $, compare the remarks after Definition \ref{defliberatedck}. 
With this in mind, we may call any quantum Cuntz-Krieger algebra of the form $ \FO(K(B,\psi)) $ a {\it quantum Cuntz algebra}.  

The algebras obtained in this way are in fact rather closely related to Cuntz algebras, as we discuss next.  

\begin{lemma} \label{QCChomo}
Let $ \FO(K(B,\psi)) $ be as above and write $ n = \dim(B) $. Then there exists a surjective $ \ast $-homomorphism $ \phi: \FO(K(B,\psi)) \rightarrow \O_n $ 
such that 
$$
\phi(S_{ij}^{(a)}) = \frac{1}{(Q_{(a)})_{ii}^{1/2} \delta} s_{ij}^{(a)}
$$
for all $ a,i,j $, where $ s_{ij}^{(a)} $ are standard generators of the Cuntz algebra $ \O_n $. 
\end{lemma}

\begin{proof} 
We just have to check that the elements $ \phi(S^{(a)}_{ij}) $ satisfies the defining relations of $ \FO(K(B,\psi)) $ from above. Indeed, we obtain
\begin{align*}
\sum_{rs} \phi(S^{(a)}_{ir}) \phi(S^{(a)}_{sr})^* \phi(S^{(a)}_{sj}) 
&= \sum_{rs} \frac{1}{(Q_{(a)})_{ii}^{1/2}(Q_{(a)})_{ss} \delta^3}s^{(a)}_{ir} (s^{(a)}_{sr})^* s^{(a)}_{sj} \\
&=\sum_{s} \frac{1}{(Q_{(a)})_{ii}^{1/2}(Q_{(a)})_{ss} \delta^3}s^{(a)}_{ij} \\
& = \frac{1}{(Q_{(a)})_{ii}^{1/2} \delta} s^{(a)}_{ij} \\
&= \phi(S_{ij}^{(a)}),
\end{align*}
and similarly
\begin{align*}
\sum_r \phi(S^{(a)}_{ri})^* \phi(S^{(a)}_{rj}) &= \sum_r \frac{1}{(Q_{(a)})_{rr} \delta^2}(s^{(a)}_{ri})^* s^{(a)}_{rj} \\
&= \delta_{ij}\\
&= \delta_{ij} \sum_{bkl} s^{(b)}_{kl} (s^{(b)}_{kl})^* \\
&= \delta_{ij} \delta^2 \sum_{bkl}  (Q_{(b)})_{kk} \phi(S^{(b)}_{kl}) \phi(S^{(b)}_{kl})^*
\end{align*}
as required. 
\end{proof}

\begin{remark}
Lemma \ref{QCChomo} implies in particular that the canonical linear map $ S: B \rightarrow \FO(K(B,\psi)) $ is injective. 
This is not always the case for general quantum Cuntz-Krieger algebras. An explicit example will be given in Example \ref{noninj} further below. 
\end{remark}

Our main structure result regarding the quantum Cuntz algebras $ \FO(K(B,\psi)) $ can be stated as follows. 

\begin{theorem} \label{quantumcompletemain} 
Let $ B $ be an $ n $-dimensional $ C^\ast $-algebra and let $ \psi: B \rightarrow \mathbb{C} $ be a $ \delta $-form satisfying $ \delta^2 \in \mathbb{N} $. 
Then $ \FO(K(B,\psi)) \cong \O_n $. 
\end{theorem} 

We will prove Theorem \ref{quantumcompletemain} using methods from the theory of quantum groups in section \ref{secquantumcomplete}. 
Under the hypothesis $ \delta^2 \in \mathbb{N} $, Theorem \ref{quantumcompletemain} implies that 
the $ \ast $-homomorphism $ \phi: \FO(K(B,\psi)) \rightarrow \O_n $ constructed in Lemma \ref{QCChomo} is an isomorphism. 

It seems remarkable that the relations defining $ \FO(K(B,\psi)) $ do indeed characterize the Cuntz algebra $ \O_n $, at least when we restrict 
to $ \delta $-forms satisfying the above integrality condition. Already in the special case $ (B, \psi) = (M_N(\mathbb{C}), \tr) $ 
from Example \ref{completematrixqgraph} it seems not even obvious that $ \FO(K(B, \psi)) $ is {\it unital}. In fact, an easy argument shows that the 
element $ e = N^2 \sum_{kl} S_{kl} (S_{kl})^* \in \FO(K(M_N(\mathbb{C}),\tr)) $ satisfies $ S_{ij} e = S_{ij} $ for all $ 1 \leq i,j \leq N $. 
In section \ref{secquantumcomplete} we will verify in particular the less evident relation $ e S_{ij} = S_{ij} $ for all $ i,j $. 

We note at the same time that $ \FO(K(M_N(\mathbb{C}),\tr)) $ is very different from the universal $ C^\ast $-algebra generated 
by the coefficients of a $ N \times N $-matrix $ S = (S_{ij}) $ satisfying $ S^* S = \id $, as the latter algebra admits many characters.

\subsection{Trivial quantum graphs} \label{parQTn}

If $ (B,\psi) $ is a finite quantum space as above, then the {\it trivial quantum graph} $ T(B,\psi) $ on $ (B,\psi) $ is given by the quantum 
adjacency matrix $ A = \id $, so that we have the matrix representation $ A^{klb}_{ija} = \delta_{ab} \delta_{ik} \delta_{jl} $. 
Using Proposition \ref{qckconcrete} we see that the quantum Cuntz-Krieger algebra $ \FO(T(B, \psi)) $ is the universal $ C^\ast $-algebra with 
generators $ S_{ij}^{(a)} $ for $ 1 \leq a \leq d, 1 \leq i,j \leq N_a $, and relations 
\begin{align*}
\sum_{kl} S_{ik}^{(a)}(S_{lk}^{(a)})^* S_{lj}^{(a)} &= S_{ij}^{(a)} \\
\sum_k (S_{ki}^{(a)})^* S_{kj}^{(a)} &= \sum_k S_{ik}^{(a)} (S_{jk}^{(a)})^* 
\end{align*}
for all $ a,i,j $. We note that $ \FO(T(B, \psi)) $ is independent of the $ \delta $-form $ \psi $ on $ B $, and we will therefore also 
write $ \FO(TB) $ instead of $ \FO(T(B,\psi)) $ in the sequel. 

\begin{example} 
Let us consider explicitly the special case of the trivial quantum graph $ TM_N = TM_N(\mathbb{C}) $ on a full matrix algebra $ B = M_N(\mathbb{C}) $. 
The $ C^\ast $-algebra $ \FO(TM_N) $ has generators $ S_{ij} $ for $ 1 \leq i,j \leq N $ satisfying the relations 
\begin{align*}
\sum_{kl} S_{ik} S^*_{lk} S_{lj} &= S_{ij} \\
\sum_k S^*_{ki} S_{kj} &= \sum_k S_{ik} S^*_{jk} 
\end{align*}
for all $ i,j $. 

It is easy to check that $ \FO(TM_N) $ maps onto Brown's algebra \cite{Brownext}, that is, the universal $ C^\ast $-algebra $ U_N^{nc} $ 
generated by the entries of a unitary $ N \times N $-matrix $ u = (u_{ij}) $, by sending $ S_{ij} $ to $ u_{ij} $. 
This shows in particular that $ \FO(TM_N) $ for $ N > 1 $ is not nuclear. 
We may also map $ \FO(TM_N) $ onto the non-unital free product $ \mathbb{C} * \cdots * \mathbb{C} $ of $ N $ copies of $ \mathbb{C} $, 
by sending $ S_{ij} $ to $ \delta_{ij} 1_i $, where $ 1_i $ denotes the unit element in the $ i $-th copy of $ \mathbb{C} $. 
It follows that $ \FO(TM_N) $ is not unital for $ N > 1 $. 

In our study of amplifications in section \ref{secamplification} we will obtain the following result on the structure of $ \FO(TM_N) $ as a special case 
of Theorem \ref{amplificationmain}. 

\begin{theorem} \label{Ktheoryquantumtrivial} 
Let $ TM_N $ be the trivial quantum graph as above. Then there exists a $ \ast $-isomorphism 
$$ 
M_N(\FO(TM_N)^+) \cong M_N(\mathbb{C}) *_1 (C(S^1) \oplus \mathbb{C}),  
$$
and the quantum Cuntz-Krieger algebra $ \FO(TM_N) $ is $ KK $-equivalent to $ C(S^1) $ for all $ N \in \mathbb{N} $. In particular 
\begin{align*}
K_0(\FO(TM_N)) &= \mathbb{Z}, \\
K_1(\FO(TM_N)) &= \mathbb{Z}. 
\end{align*}
\end{theorem} 

Here $ *_1 $ denotes the unital free product and $ \FO(TM_N)^+ $ is the minimal unitarization of $ \FO(TM_N) $. 

With little extra effort one can also determine generators for the $ K $-groups in Theorem \ref{Ktheoryquantumtrivial}. More precisely, if we 
write $ S = (S_{ij}) $ for the matrix of generators of $ \FO(TM_N) $, then these are represented by the projection $ S^* S \in M_N(\FO(TM_N)) $ 
and the unitary $ S - (1 - S^* S) \in M_N(\FO(TM_N)^+) $, respectively. 
\end{example} 

\begin{remark} 
Combining Theorem \ref{Ktheoryquantumtrivial} and Proposition \ref{sumfreeproduct} below one can determine 
the $ K $-theory of $ \FO(TB) $ for general $ B $. More precisely, if $ B \cong \bigoplus_{a = 1}^d M_{N_a}(\mathbb{C}) $ then we obtain 
\begin{align*}
K_0(\FO(TB)) &= \mathbb{Z}^d, \\
K_1(\FO(TB)) &= \mathbb{Z}^d,
\end{align*}
taking into account \cite{Cuntzfreeproduct}. 
\end{remark}

\subsection{Diagonal quantum graphs} \label{pardiagonal}

A natural generalization of the trivial quantum graphs described in the previous paragraph are the {\it diagonal quantum graphs}. Here, we take $ (B,\psi) $ 
again to be an arbitrary finite quantum space in standard form, but replace the trivial quantum adjacency matrix $ A = \id $ with a map of the form 
$$
A(f_{ij}^{(a)}) = x_{ij}^{(a)} f_{ij}^{(a)} 
$$ 
for some suitable complex numbers $ x^{(a)}_{ij} \in \mathbb{C} $ for $ 1 \leq a \leq d, 1 \leq i,j \leq N_a $. Note that if $ B $ is abelian then the 
associated adjacency matrix is a diagonal matrix with entries in $ \{0,1 \} $. That is, the only edges possible are self-loops, and we recover precisely 
the classical notion of a diagonal graph.  

In the non-commutative setting the notion of a diagonal graph is somewhat richer. Namely, Lemma \ref{adjacencycomputation} shows that the only requirements 
on the coefficients $ x_{ij}^{(a)} $ are 
$$
\sum_s (Q_{(b)}^{-1})_{ss} x_{ks}^{(b)} x_{sl}^{(b)} = \delta^2 x_{kl}^{(b)} 
$$
for all $ 1 \leq b \leq d, 1 \leq k,l \leq N_b $.

\begin{example} \label{noninj}
Let $ B = M_N(\mathbb{C}) $ be equipped with the $ \delta $-form $ \psi $ corresponding to the 
diagonal matrix $ Q $ with entries $ q_1, \dots, q_N $ satisfying $ q_1 + \cdots + q_N = 1 $. Moreover let $ A $ be the diagonal quantum adjacency matrix with 
coefficients $ A^{ij}_{kl} = x_{ij} \delta_{ik} \delta_{jl} $ for some scalars $ x_{ij} $ satisfying $ \sum_s q_s^{-1} x_{ks} x_{sl} = \delta^2 x_{kl} $. 
The quantum Cuntz-Krieger algebra $ \FO(\G) $ associated with the diagonal quantum graph $ \G = (B, \psi, A) $ has generators $ S_{ij} $ for $ 1 \leq i,j \leq N $
satisfying the relations 
\begin{align*}
\sum_{kl} S_{ik} S_{lk}^* S_{lj} &= S_{ij} \\
\sum_k S^*_{ki} S_{kj} &= \sum_k x_{ij} S_{ik} S^*_{jk} 
\end{align*}
for all $ i,j $. 

Consider the special case $ x_{11} = q_1 \delta^2 $ and $ x_{ij} = 0 $ else. From the second relation above we 
get $ \sum_i S^*_{ij} S_{ij} = 0 $ for $ j > 1 $, and hence $ S_{ij} = 0 $ for all $ 1 \leq i \leq N $ and $ j > 1 $. 
This shows that the canonical linear map $ S: B \rightarrow \FO(\G) $ in the definition of a quantum Cuntz-Krieger algebra need not be injective. 

One may interpret this as a reflection of the fact that we work with rather general quantum adjacency matrices. It would be 
interesting to identify a suitable condition on directed quantum graphs $ \G $ which ensures that the map $ S: B \rightarrow \FO(\G) $ is injective. 

Note also that we have $ \sum_l S_{i1} S_{l1}^* S_{l1} = S_{i1} $ and $ \sum_k S^*_{k1} S_{k1} = x_{11} S_{11} S^*_{11} $ in the above special case. 
Hence for all complex numbers $ \epsilon_1, \dots, \epsilon_N $ satisfying $ |\epsilon_1|^2 + \cdots + |\epsilon_N|^2 = 1 $ 
and $ x_{11} |\epsilon_1|^2 = 1 $ there exists a $ \ast $-homomorphism $ \epsilon: \FO(\G) \rightarrow \mathbb{C} $ satisfying
$$ 
\epsilon(S_{ij}) = 
\begin{cases} 
\epsilon_i & j = 1 \\
0 & j > 1. 
\end{cases} 
$$
It follows in particular that the $ C^\ast $-algebra $ \FO(\G) $ admits a trace. 
\end{example}

\subsection{Direct sums and tensor products of quantum graphs} \label{parsumtensor}

Assume that $ \G_1 = (B_1, \psi_1, A_1) $ and $ \G_2 = (B_2, \psi_2, A_2) $ are directed quantum graphs. We obtain a finite quantum space structure  
on the direct sum $ B_1 \oplus B_2 $ by considering the state 
$$ 
\psi = \frac{\delta_1^2}{\delta^2} \psi_1 \oplus \frac{\delta_2^2}{\delta^2} \psi_2, 
$$
with $ \delta^2 = \delta_1^2 + \delta_2^2 $. 
It is easy to check that $ A = A_1 \oplus A_2 $ defines a quantum adjacency matrix on $ (B_1 \oplus B_2, \psi) $, 
so that $ \G_1 \oplus \G_2 = (B_1 \oplus B_2, \psi, A) $ is a directed quantum graph. Classically, this construction corresponds to taking the 
disjoint union of graphs. 

\begin{prop} \label{sumfreeproduct}
Let $ \G_1 = (B_1, \psi_1, A_1) $ and $ \G_2 = (B_2, \psi_2, A_2) $ be directed quantum graphs. Then 
$$ 
\FO(\G_1 \oplus \G_2) \cong \FO(\G_1) * \FO(\G_2)
$$
is the non-unital free product of $ \FO(\G_1) $ and $ \FO(\G_2) $. 
\end{prop} 

\begin{proof} 
This follows directly from the universal properties of the algebras involved, noting that the quantum adjacency matrix $ A_1 \oplus A_2 $ 
does not mix generators from $ B_1 $ and $ B_2 $. 
\end{proof}

We can also form tensor products in a natural way. If $ \G_1 = (B_1, \psi_1, A_1) $ and $ \G_2 = (B_2, \psi_2, A_2) $ are directed quantum 
graphs then $ \psi = \psi_1 \otimes \psi_2 $ is a $ \delta $-form on the tensor product $ B_1 \otimes B_2 $ with $ \delta = \delta_1 \delta_2 $. 
Moreover $ A = A_1 \otimes A_2 $ defines a quantum adjacency matrix on $ (B_1 \otimes B_2, \psi) $. 
We let $ \G_1 \otimes \G_2 $ be the corresponding directed quantum graph. 

Compared to the case of direct sums, it seems less obvious how to describe the structure of $ \FO(\G_1 \otimes \G_2) $ in terms of $ \FO(\G_1) $ 
and $ \FO(\G_2) $ in general. We shall discuss a special case in the next section.

\section{Amplification} \label{secamplification}

In this section we study examples of quantum Cuntz-Krieger algebras obtained from classical graphs by replacing 
the vertices with matrix blocks. This \emph{amplification} procedure is a special case of the 
tensor product construction for quantum graphs described in paragraph \ref{parsumtensor}. 

Given a directed quantum graph $ \G = (B, \psi, A) $ and $ N \in \mathbb{N} $ we define the amplification $ M_N(\G) $ of $ \G $ to be the tensor 
product $ M_N(\G) = \G \otimes TM_N $, where $ TM_N $ is the trivial quantum graph on $ M_N(\mathbb{C}) $ 
as defined in paragraph \ref{parQTn}. Explicitly, $ M_N(\G) $ is the directed quantum graph with underlying $ C^\ast $-algebra $ B \otimes M_N(\mathbb{C}) $, 
state $ \phi = \psi \otimes \tr $, and quantum adjacency matrix $ A^{(N)} = A \otimes \id $. 

In the sequel we shall restrict ourselves to the case that $ \G $ is associated with a classical graph. Recall from paragraph \ref{parclassical} 
that if $ E = (E^0, E^1) $ is a simple finite directed classical graph then the adjacency matrix $ B_E $ of $ E $ induces canonically a directed 
quantum graph structure on $ C(E^0) $ with its unique $ \delta $-form. 

\begin{lemma} \label{matrixampqgraphrelations}
Let $ E = (E^0, E^1) $ be a simple finite directed classical graph and denote by $ \G = (C(E^0), \tr, B_E) $ the directed quantum graph 
corresponding to $ E $. Then the quantum Cuntz-Krieger algebra $ \FO(M_N(\G)) $ associated with the amplification $ M_N(\G) $ is 
the universal $ C^\ast $-algebra with generators $ S_{eij} $ for $ e \in E^0 $ and $ 1 \leq i,j \leq N $, satisfying the relations 
\begin{align*}
\sum_{rs} S_{eir} S^*_{esr} S_{esj} &= S_{eij} \\
\sum_k S^*_{eki} S_{ekj} &= \sum_k \sum_{f \in E^0} B_E(e,f) S_{fik} S_{fjk}^*.  
\end{align*}
\end{lemma} 

\begin{proof} 
Consider the generators $ S_{eij} = S(f^{(e)}_{ij}) $ in $ \FO(M_N(\G)) $ associated with the 
adapted matrix units $ f^{(e)}_{ij} = n N \delta_e \otimes e_{ij} $, where $ e \in E^0 $ and $ n $ is the number of vertices of $ E $. 
Noting that the quantum adjacency matrix of $ M_N(\G) $ is given by 
$$ 
A^{(N)}(f^{(e)}_{ij}) = \sum_{f \in E^0} B_E(e,f) f^{(f)}_{ij}, 
$$
the assertion is a direct consequence of Proposition \ref{qckconcrete}. 
\end{proof}

We will follow arguments of McClanahan \cite{McClanahanunitarymatrix} to study the structure of the quantum Cuntz-Krieger algebras 
in Lemma \ref{matrixampqgraphrelations}. As a first step we discuss a slight strengthening of Theorem 2.3 in \cite{McClanahanunitarymatrix}. 
If $ A $ is a $ C^\ast $-algebra we write $ A^+ $ for the unital $ C^\ast $-algebra obtained by adjoining an identity element to $ A $, 
and if $ A,B $ are unital $ C^\ast $-algebras we denote by $ A *_1 B $ their unital free product. 

\begin{prop} \label{unitalfreeproductkk}
Let $ A $ be a separable $ C^\ast $-algebra. Then $ M_N(\mathbb{C}) *_1 A^+ $ is $ KK $-equivalent to $ A^+ $. 
\end{prop} 

\begin{proof} 
This fact is certainly known to experts, but we shall give the details for the convenience of the reader. 

Note first that $ A^+ $ is $ KK $-equivalent to the direct sum $ A \oplus \mathbb{C} $. 
This equivalence is implemented by taking the direct sum of the canonical $ * $-homomorphisms $ A \rightarrow A^+ $ 
and $ \mathbb{C} \rightarrow A^+ $ at the level of $ KK $-theory. 

We consider the unital $ * $-homomorphism $ \phi: M_N(\mathbb{C}) *_1 A^+ \rightarrow M_N(\mathbb{C}) \otimes A^+ $ given by 
$$ 
\phi(e_{ij}) = e_{ij} \otimes 1, \qquad \phi(a) = e_{11} \otimes a, 
$$ 
for $ 1 \leq i,j \leq N $ and $ a \in A $, and view this as a class $ [\phi] \in KK(M_N(\mathbb{C}) *_1 A^+, A^+) $. 

In the opposite direction we define a map $ \psi_A: A \rightarrow M_N(\mathbb{C}) \otimes (M_N(\mathbb{C}) *_1 A^+) $ by 
$$ 
\psi_A(a) = \sum_{kl} e_{kl} \otimes e_{1k} a e_{l1}. 
$$
Then 
\begin{align*}
\psi_A(a) \psi_A(b) &= \sum_{klrs} e_{kl} e_{rs} \otimes e_{1k} a e_{l1} e_{1r} b e_{s1} \\
&= \sum_{kls} e_{ks} \otimes e_{1k} a e_{l1} e_{1l} b e_{s1} \\
&= \sum_{ks} e_{ks} \otimes e_{1k} a b e_{s1} = \psi_A(ab) 
\end{align*} 
and $ \psi_A(a^*) = \psi_A(a)^* $, so that the map $ \psi_A $ is a $ * $-homomorphism. 

Consider also the $ * $-homomorphism $ \psi_\mathbb{C}: \mathbb{C} \rightarrow M_N(\mathbb{C}) \otimes (M_N(\mathbb{C}) *_1 A^+) $ 
given by $ \psi_\mathbb{C}(1) = e_{11} \otimes e_{11} $. Combining the maps $ \psi_A $ and $ \psi_\mathbb{C} $, and using that $ A^+ $ is $ KK $-equivalent 
to $ A \oplus \mathbb{C} $, we obtain a class in $ KK(A^+, M_N(\mathbb{C}) *_1 A^+) $, which we shall denote by $ [\psi] $. 

We claim that the classes $ [\phi] $ and $ [\psi] $ are mutually inverse. 
In order to determine the composition $ [\phi] \circ [\psi] \in KK(A^+, A^+) $ it suffices to 
compute $ M_N(\phi) \circ \psi_A $ and $ M_N(\phi) \circ \psi_\mathbb{C} $, respectively. 

We calculate
$$ 
(M_N(\phi) \circ \psi_A)(a) = \sum_{kl} e_{kl} \otimes \phi(e_{1k} a e_{l1}) = \sum_{kl} e_{kl} \otimes e_{1k} e_{11} e_{l1} \otimes a 
= e_{11} \otimes e_{11} \otimes a 
$$ 
for $ a \in A $ and $ (M_N(\phi) \circ \psi_\mathbb{C})(1) = M_N(\phi)(e_{11} \otimes e_{11}) = e_{11} \otimes e_{11} \otimes 1 $. 
This immediately yields $ [\phi] \circ [\psi] = \id $. 

Next consider $ [\psi] \circ [\phi] \in KK(M_N(\mathbb{C}) *_1 A^+, M_N(\mathbb{C}) *_1 A^+) $. 
Let us write $ j_{A^+}: A^+ \rightarrow M_N(\mathbb{C}) *_1 A^+ $ and $ j_{M_N(\mathbb{C})}: M_N(\mathbb{C}) \rightarrow M_N(\mathbb{C}) *_1 A^+ $ 
for the canonical inclusion homomorphisms. Moreover write $ u: \mathbb{C} \rightarrow M_N(\mathbb{C}) \oplus A^+ $ for the 
unit map. According to \cite{Germainfreeproduct}, \cite{FimaGermainamalgamated},  
the suspension of $ M_N(\mathbb{C}) *_1 A^+ $ is $ KK $-equivalent to the cone of $ u $. In order to show $ [\psi] \circ [\phi] = \id $ it therefore suffices 
to verify $ [\psi] \circ [\phi] \circ [j_{A^+}] = [j_{A^+}] $ and $ [\psi] \circ [\phi] \circ [j_{M_N(\mathbb{C})}] = [j_{M_N(\mathbb{C})}] $. 

We calculate 
$$ 
(M_N(\psi_A) \circ \phi)(a) = M_N(\psi)(e_{11} \otimes a) = \sum_{kl} e_{11} \otimes e_{kl} \otimes e_{1k} a e_{l1} 
$$ 
for $ a \in A $. 
Pick a continuous path of unitaries $ U_t $ in $ M_N(\mathbb{C}) \otimes M_N(\mathbb{C}) $ such that $ U_0 = \id $ and 
$$ 
U_1(e_k \otimes e_1) = e_1 \otimes e_k 
$$ 
for all $ k $, and push this into the last two tensor factors of $ M_N(\mathbb{C}) \otimes M_N(\mathbb{C}) \otimes (M_N(\mathbb{C}) *_1 A^+) $ via the 
obvious map. Then conjugating $ (M_N(\psi_A) \circ \phi)(a) $ by $ U_1 $ gives $ e_{11} \otimes e_{11} \otimes a $ for all $ a \in A $. 
It follows that $ [\psi] \circ [\phi] \circ [j_A] = [j_A] $, where we write $ j_A $ for the restriction of $ j_{A^+} $ to $ A \subset A^+ $. 

Next we calculate 
$$ 
(M_N(\psi_\mathbb{C}) \circ \phi)(1) = M_N(\psi_\mathbb{C})(1 \otimes 1) = 1 \otimes e_{11} \otimes e_{11}.  
$$ 
Conjugating this with the unitary $ U_1 $ from above, pushed into the first and third tensor factors, gives $ e_{11} \otimes e_{11} \otimes 1 $. 
Hence $ [\psi] \circ [\phi] \circ [j_\mathbb{C}] = [j_\mathbb{C}] $, where $ j_\mathbb{C} $ denotes the restriction of $ j_{A^+} $ 
to $ \mathbb{C} \subset A^+ $. Combining these two observations gives $ [\psi] \circ [\phi] \circ [j_{A^+}] = [j_{A^+}] $. 

Finally, we have 
$$ 
(M_N(\psi_\mathbb{C}) \circ \phi)(e_{ij}) = M_N(\psi_\mathbb{C})(e_{ij} \otimes 1) = e_{ij} \otimes e_{11} \otimes e_{11}, 
$$ 
so that conjugating $ (M_N(\psi_\mathbb{C}) \circ \phi)(e_{ij}) $ by $ U_1 $ in the first and third tensor factors 
gives $ e_{11} \otimes e_{11} \otimes e_{ij} $ for all $ i,j $. 
We conclude $ [\psi] \circ [\phi] \circ [j_{M_N(\mathbb{C})}] = [j_{M_N(\mathbb{C})}] $, and this finishes the proof. 
\end{proof} 

With these preparations in place let us now present our main result on amplified quantum Cuntz-Krieger algebras. 

\begin{theorem} \label{amplificationmain}
Assume that $ E = (E^0, E^1) $ is a finite directed simple graph and let $ \G = (C(E^0), \tr, B_E) $ be the corresponding 
directed quantum graph. Then the following holds. 
\begin{bnum} 
\item[a)] We have an isomorphism $ M_N(\FO(M_N(\G))^+) \cong M_N(\mathbb{C}) *_1 (\FO(\G)^+) $. 
\item[b)] $ \FO(M_N(\G)) $ is $ KK $-equivalent to the classical Cuntz-Krieger algebra $ \O_{B_E} $. 
\end{bnum} 
\end{theorem} 

\begin{proof} 
$ a) $ In the sequel we shall write $ C = M_N(\mathbb{C}) *_1 (\FO(\G)^+) $ and $ D = \FO(M_N(\G)) $. 

We define a $ \ast $-homomorphism $ g: D \rightarrow C $ by 
$$
g(S_{eij}) = \sum_k e_{ki} S_e e_{jk}  
$$
on generators. To check that this is well-defined we use Lemma \ref{matrixampqgraphrelations} to calculate 
\begin{align*}
\sum_{kl} g(S_{eik}) g(S_{elk})^* g(S_{elj}) &= \sum_{rstkl} e_{ri} S_e e_{kr} e_{sk} S_e^* e_{ls} e_{tl} S_e e_{jt} \\
&= \sum_{rkl} e_{ri} S_e e_{kk} S_e^* e_{ll} S_e e_{jr} \\
&= \sum_r e_{ri} S_e S_e^* S_e e_{jr} \\
&= \sum_r e_{ri} S_e e_{jr} \\
&= g(S_{eij}) 
\end{align*}
and 
\begin{align*}
\sum_k g(S_{eki})^* g(S_{ekj}) &= \sum_{rsk} e_{ri} S_e^* e_{kr} e_{sk} S_e e_{js} \\
&= \sum_r e_{ri} S_e^* S_e e_{jr} \\
&= \sum_r \sum_{f \in E^0} B_E(e,f) e_{ri} S_f S_f^* e_{jr} \\
&= \sum_k \sum_{f \in E^0} B_E(e,f) g(S_{fik}) g(S_{fjk})^*
\end{align*}
for $ e \in E^0 $ and $ 1 \leq i,j \leq N $. 

Let $ g^+: D^+ \rightarrow C $ be the unital extension of $ g $. It is easy to see that the image of $ g^+ $ is contained in the relative 
commutant $ M_N(\mathbb{C})' $ of $ M_N(\mathbb{C}) $ inside the free product. In fact, we have 
$$ 
g(S_{eij}) e_{kl} = \sum_r e_{ri} S_e e_{jr} e_{kl} = e_{ki} S_e e_{jl} = \sum_r e_{kl} e_{ri} S_e e_{jr} = e_{kl} g(S_{eij})  
$$ 
for all $ i,j,k,l $. 
We can thus extend $ g^+ $ to a unital $ \ast $-homomorphism $ G: M_N(D^+) \rightarrow C $ by setting $ G(e_{ij}) = e_{ij} $ 
and $ G(x) = g(x) $ for $ x \in D^+ $. 

Let us also define a unital $ \ast $-homomorphism $ F: C \rightarrow M_N(D^+) = D^+ \otimes M_N(\mathbb{C}) $ by 
\begin{align*}
F(e_{ij}) &= 1 \otimes e_{ij} \\
F(S_e) &= \sum_{ij} S_{eij} \otimes e_{ij}. 
\end{align*}
To see that this is well-defined we only need to check that these formulas define unital $ \ast $-homomorphisms from $ M_N(\mathbb{C}) $ 
and $ \FO(\G)^+ $ into $ M_N(D^+) $, respectively. For $ M_N(\mathbb{C}) $ this is obvious. For $ \FO(\G)^+ $ we need to check the free Cuntz-Krieger 
relations for the elements $ F(S_e) $. In fact, each $ F(S_e) $ is a partial isometry by construction, and using Lemma \ref{matrixampqgraphrelations} we calculate
\begin{align*}
F(S_e)^* F(S_e) &= \sum_{ijkl} S_{eij}^* S_{ekl} \otimes e_{ji} e_{kl} \\
&= \sum_{ijl} S_{eij}^* S_{eil} \otimes e_{jl} \\
&= \sum_{ijl} \sum_{f \in E^0} B_E(e,f) S_{fji} S_{fli}^* \otimes e_{jl} \\
&= \sum_{f \in E^0} B_E(e,f) \sum_{ijkl} S_{fji} S_{flk}^* \otimes e_{ji} e_{kl} \\
&= \sum_{f \in E^0} B_E(e,f) F(S_f) F(S_f)^* 
\end{align*}  
as required. 

Next observe that $ F \circ G: M_N(D^+) \rightarrow M_N(D^+) $ satisfies 
\begin{align*} 
(F \circ G)(S_{eij} \otimes 1) &= \sum_k F(e_{ki}) F(S_e) F(e_{jk}) \\
&= \sum_{krs} (1 \otimes e_{ki}) (S_{ers} \otimes e_{rs}) (1 \otimes e_{jk}) = S_{eij} \otimes 1 
\end{align*}
for all $ e \in E^0 $ and $ (F \circ G)(e_{ij}) = e_{ij} $ for all $ i,j $. This implies $ F \circ G = \id $. 
Similarly, we have 
\begin{align*} 
(G \circ F)(S_e) &= \sum_{ij} G(S_{eij} \otimes e_{ij}) = \sum_{ij} e_{ii} S_e e_{jj} = S_e 
\end{align*}
for all $ e \in E^0 $, and $ (G \circ F)(e_{ij}) = e_{ij} $ for all $ i,j $. We conclude that $ G \circ F = \id $. 

$ b) $ Clearly $ M_N(\FO(M_N(\G))^+) $ is $ KK $-equivalent to $ \FO(M_N(\G))^+ $. 
According to Proposition \ref{unitalfreeproductkk}, we also know that $ M_N(\mathbb{C}) *_1 (\FO(\G)^+) $ is $ KK $-equivalent to $ \FO(\G)^+ $. 
It is easy to check that these equivalences are both compatible with the canonical augmentation morphisms to $ \mathbb{C} $. 
Hence $ \FO(M_N(\G)) $ is $ KK $-equivalent to $ \FO(\G) $. Finally, recall from Theorem \ref{graphversusfreegraph} that the free 
Cuntz-Krieger algebra $ \FO(\G) = \FO_{B_E} $ is $ KK $-equivalent to $ \O_{B_E} $. 
\end{proof}

Under some mild extra assumptions, Theorem \ref{amplificationmain} allows one to compute the $ K $-theory of $ \FO(M_N(\G)) $ in terms of the graph $ E $, 
see \cite{Cuntzmarkov2} and chapter 7 in \cite{Raeburngraph}. 

Finally, remark that if $ E $ is the graph with one vertex and one self-loop then we have $ \FO(\G) = \FO_{B_E} = C(S^1) $, 
and $ \FO(M_N(\G)) = \FO(TM_N) $ is the quantum Cuntz-Krieger algebra of the trivial quantum graph on $ M_N(\mathbb{C}) $. 
Therefore Theorem \ref{amplificationmain} implies Theorem \ref{Ktheoryquantumtrivial}.

\section{Quantum symmetries of quantum Cuntz-Krieger algebras} \label{secquantumsymmetry} 

In this section we study how quantum symmetries and quantum isomorphisms of directed quantum graphs induce symmetries of their associated 
quantum Cuntz-Krieger algebras. This will be useful in particular to exhibit relations between the $ C^\ast $-algebras corresponding to quantum 
isomorphic quantum graphs. 

\subsection{Gauge actions} \label{pargauge}

Before discussing quantum symmetries, let us first show that there is a canonical gauge action on quantum Cuntz-Krieger algebras, thus 
providing very natural classical symmetries. This is analogous to the well-known gauge action on Cuntz-Krieger algebras and graph $ C^\ast $-algebras, which plays 
a crucial role in the analysis of the structure of these $ C^\ast $-algebras, compare \cite{Raeburngraph}. 

Let $ \G = (B, \psi, A) $ be a directed quantum graph, and let $ \FO(\G) $ be the corresponding quantum Cuntz-Krieger algebra. 
For $ \lambda \in U(1) $ consider the linear map $ S_\lambda: B \rightarrow \FO(\G) $ given by 
$$ 
S_\lambda(b) = \lambda S(b), 
$$
where $ S: B \rightarrow \FO(\G) $ is the canonical linear map. Then we 
have $ S_\lambda^*(b) = (\lambda S(b^*))^* = \overline{\lambda} S^*(b) $ for all $ b \in B $, 
and using this relation it is easy to check that $ S_\lambda: B \rightarrow \FO(\G) $ is a quantum Cuntz-Krieger $ \G $-family. 
By the universal property of $ \FO(\G) $ we obtain a corresponding automorphism $ \alpha_\lambda \in \Aut(\FO(\G)) $, 
and these automorphisms combine to a strongly continuous action of $ U(1) $ on $ \FO(\G) $. 

In terms of the generators of $ \FO(\G) $ as in Proposition \ref{qckconcrete} the gauge action is given by 
$$
\alpha_\lambda(S^{(a)}_{ij}) = \lambda S^{(a)}_{ij},  
$$
from which it is easy to determine the action on arbitrary noncommutative polynomials in the generators, and the 
decomposition into spectral subspaces. 

In some cases one may define more general gauge type actions. For instance, for the complete quantum graph $ K(M_N(\mathbb{C}), \tr) $ from paragraph \ref{parQKn} 
and the trivial quantum graph $ TM_N $ from paragraph \ref{parQTn} we have an action of $ U(1) \times U(1)^N $, given by 
$$
\alpha_{\lambda \mu}(S_{ij}) = \lambda \frac{\mu_i}{\mu_j} \,S_{ij} 
$$
on generators. In fact, one may even extend this to an action of $ U(1) \times U(N) $ by setting
$$
\alpha_{\lambda U}(S) = \lambda U S U^*, 
$$
where $ S = (S_{ij}) $ is the generating matrix partial isometry. 

However, none of the above actions seems to suffice to obtain structural information about quantum Cuntz-Krieger algebras in the same way 
as for classical graph algebras. In particular, the corresponding fixed point algebras tend to have a more complicated structure than in 
the classical setting. 

It turns out that this deficiency can be compensated to some extent by considering actions of compact quantum groups instead, and in 
particular symmetries arising from suitable monoidal equivalences between quantum automorphism groups of directed quantum graphs. We will explain these  
constructions in the following paragraphs.

\subsection{Compact quantum groups} 

Let us first give a quick review of the definition of compact quantum groups and their action on $ C^\ast $-algebras. For more background and further information 
we refer to \cite{Woronowiczleshouches}, \cite{NTlecturenotes}. 

A compact quantum group $ G $ is given by a unital $ C^\ast $-algebra $ C(G) $ together with a 
unital $ \ast $-homomorphism $ \Delta: C(G) \rightarrow C(G) \otimes C(G) $ 
such that $ (\Delta \otimes \id) \Delta = (\id \otimes \Delta) \Delta $ and the density conditions
$$ 
[\Delta(C(G)) (C(G) \otimes 1)] = C(G) \otimes C(G) = [\Delta(C(G)) (1 \otimes C(G))] 
$$
hold. 

We will mainly work with the canonical dense Hopf $ \ast $-algebra $ \Poly(G) \subset C(G) $, consisting of the matrix coefficients 
of all finite dimensional unitary representations of $ G $. For the definition of unitary representations and their intertwiners see \cite{NTlecturenotes}. 
The collection of all finite dimensional unitary representations of $ G $ forms naturally a $ C^\ast $-tensor category $ \Rep(G) $. 

On the $ C^\ast $-level we will only consider the universal completions of $ \O(G) $ in the sequel, and always denote them by $ C(G) $. 
With this in mind, a morphism $ H \rightarrow G $ of compact quantum groups is nothing but a $ \ast $-homomorphism $ C(G) \rightarrow C(H) $ 
compatible with the comultiplications. Equivalently, such a morphism is given by a homomorphism $ \Poly(G) \rightarrow \Poly(H) $ of 
Hopf $ \ast $-algebras. One says that $ H $ is a quantum subgroup of $ G $ if there exists a morphism $ H \rightarrow G $ 
such that the corresponding homomorphism of Hopf $ \ast $-algebras $ \Poly(G) \rightarrow \Poly(H) $ is surjective. 

By definition, an action of a compact quantum group $ G $ on a $ C^\ast $-algebra $ A $ is a $ \ast $-homomorphism $ \alpha: A \rightarrow A \otimes C(G) $ 
satisfying $ (\alpha \otimes \id)\alpha = (\id \otimes \Delta)\alpha $ and the density condition $ [(1 \otimes C(G)) \alpha(A)] = A \otimes C(G) $. 
A $ C^\ast $-algebra $ A $ equipped with an action of $ G $ will also be called a $ G $-$ C^\ast $-algebra. 
Every $ G $-$ C^\ast $-algebra $ A $ contains a canonical dense $ \ast $-subalgebra $ \A \subset A $, given by 
the algebraic direct sum of the spectral subspaces of the action. Moreover, the map $ \alpha $ restricts to 
a $ \ast $-homomorphism $ \alpha: \A \rightarrow \A \otimes \Poly(G) $, and this defines an algebra coaction in the sense of Hopf algebras. 
In particular, one has $ (\id \otimes \epsilon)\alpha(a) = a $ for all $ a \in \A $, where $ \epsilon: \Poly(G) \rightarrow \mathbb{C} $ 
is the counit. 

If $ A $ is a $ G $-$ C^\ast $-algebra then the fixed point subalgebra of $ A $ is defined by 
$$ 
A^G = \{a \in A \mid \alpha(a) = a \otimes 1 \},  
$$
and a unital $ G $-$ C^\ast $-algebra $ A $ is called ergodic if $ A^G = \mathbb{C} 1 $. 
The same terminology is also used for $ \ast $-algebras equipped with algebra coactions of $ \Poly(G) $. 

Let us now review the definition of quantum automorphism groups of finite quantum spaces in the sense of Definition \ref{deffqs}. 
These quantum groups were introduced by Wang \cite{Wangqsymmetry} and studied further by Banica \cite{Banicageneric} and others. 
If $ G $ is a compact quantum group and $ \omega: A \rightarrow \mathbb{C} $ a state on a $ G $-$ C^\ast $-algebra $ A $ with 
action $ \alpha: A \rightarrow A \otimes C(G) $, then we say that the action preserves $ \omega $ if 
$$
(\omega \otimes \id) \alpha(a) = \omega(a) 1 
$$
for all $ a \in A $.  

\begin{definition} \label{defqut}
Let $ (B, \psi) $ be a finite quantum space. The quantum automorphism group of $ (B, \psi) $ is the universal compact quantum group $ G^+(B,\psi) $ 
equipped with an action $ \beta: B \rightarrow B \otimes C(G^+(B, \psi)) $ which preserves $ \psi $. 
\end{definition} 

In other words, if $ G $ is a compact quantum group and $ \gamma: B \rightarrow B \otimes C(G) $ an action of $ G $ preserving $ \psi $, then 
there exists a unique $ * $-homomorphism $ \pi: C(G^+(B, \psi)) \rightarrow C(G) $, compatible with the comultiplications, such that the diagram 
$$
\xymatrix{
B \ar@{->}[r]^{\!\!\!\!\!\!\!\!\!\!\!\!\!\!\!\!\!\!\!\!\!\! \beta} \ar@{->}[rd]_{\gamma} & B \otimes C(G^+(B, \psi)) \ar@{->}[d]^{\id \otimes \pi} \\ 
& B \otimes C(G) 
}
$$
is commutative. 

The most prominent example of a quantum automorphism group is the quantum permutation group $ S_N^+ $. This is the quantum automorphism 
group of $ B = \mathbb{C}^N $ with its unique $ \delta $-form. The corresponding $ C^\ast $-algebra $ C(S_N^+) = C(G^+(\mathbb{C}^N, \tr)) $ 
is the universal $ C^\ast $-algebra generated by projections $ u_{ij} $ for $ 1 \leq i,j \leq N $ such that 
$$
\sum_k u_{ik} = 1 = \sum_k u_{kj} 
$$
for all $ i,j $. These relations can be phrased by saying that the matrix $ u = (u_{ij}) $ is a magic 
unitary. The comultiplication $ \Delta: C(S_N^+) \rightarrow C(S_N^+) \otimes C(S_N^+) $ is defined by
$$
\Delta(u_{ij}) = \sum_{k = 1}^n u_{ik} \otimes u_{kj} 
$$
on generators. 

\begin{remark} 
Quantum automorphism groups can always be described explicitly in terms of generators and relations, see Proposition 2.10 in \cite{Mrozinskiso3deformations}. 
More precisely, let us assume that $ (B, \psi) $ is a finite quantum space in standard form as in section \ref{parqgraph}, so that 
$$ 
B = \bigoplus_{a = 1}^d M_{N_a}(\mathbb{C}), \qquad \psi(x) = \sum_{a = 1}^d \Tr(Q_{(a)} x_a)
$$
for $ x = (x_1, \dots, x_d) \in B $. Then the Hopf $ \ast $-algebra $ \Poly(G^+(B, \psi)) $ is generated 
by elements $ v_{ija}^{rsb} $ for $ 1 \leq a, b \leq d $ and $ 1 \leq i,j \leq N_a, 1 \leq r,s \leq N_b $, satisfying the relations 

\begin{itemize}
\item[(A1a)] $ \sum_w v_{kla}^{xwc} v_{rsb}^{wyc} = \delta_{ab} \delta_{lr} v_{ksa}^{xyc} $
\item[(A1b)] $ \sum_w (Q_{(c)})^{-1}_{ww} v^{srb}_{ywc} v^{lka}_{wxc} = \delta_{lr} \delta_{ab} (Q_{(a)})^{-1}_{ll} v^{ska}_{yxc} $
\item[(A2)] $ (v_{kla}^{xyc})^* = v_{lka}^{yxc} $
\item[(A3a)] $ \sum_{xb} (Q_{(b)})_{xx} v^{xxb}_{kla} = \delta_{kl} (Q_{(a)})_{kk} $
\item[(A3b)] $ \sum_{ka} v^{xyb}_{kka} = \delta_{xy} $.
\end{itemize}

In terms of the standard matrix units $ e^{(a)}_{ij} $ for $ B $ and the generators $ v_{ija}^{rsb} $, the defining 
action $ \beta: B \rightarrow B \otimes \Poly(G^+(B,\psi)) $ is given by 
$$ 
\beta(e^{(a)}_{ij}) = \sum_{bkl} e^{(b)}_{kl} \otimes v_{ija}^{klb}, 
$$
and the matrix $ v = (v_{ija}^{rsb}) $ is also called the fundamental matrix of $ G^+(B, \psi) $. 

We will reobtain the above description of the $ \ast $-algebra $ \Poly(G^+(B, \psi)) $ as a special case of Proposition \ref{qisorelations} below. 
\end{remark}

\subsection{Quantum symmetries of quantum graphs} \label{parqsym}

In this paragraph we discuss the quantum automorphism group of a directed quantum graph, and also quantum isomorphisms relating a pair
of directed quantum graphs. 

Recall first that if $ E = (E^0, E^1) $ is a simple finite 
graph then the automorphism group $ \Aut(E) $ consists of all bijections of $ E^0 $ which preserves the adjacency relation in $ E $. 
If $ |E^0| = N $ and $ A \in M_N(\mathbb{Z}) $ is the adjacency matrix of $ E $, then this can be expressed as 
$$
\Aut(E) = \{\sigma \in S_N \mid \sigma A = A \sigma\} \subset S_N, 
$$
where one views elements of the symmetric group as permutation matrices. In \cite{Banicaqutgraph}, Banica defined the quantum automorphism group $ G^+(E) $ 
of the graph $ E $ via the $ C^\ast $-algebra
$$
C(G^+(E)) = C(S_N^+)/\langle u A = A u \rangle, 
$$
where $ u = (u_{ij}) \in M_N(C(S_N^+)) $ denotes the defining magic unitary matrix. This yields a quantum subgroup of $ S_N^+ $, which contains the 
classical automorphism group $ \Aut(E) $ as a quantum subgroup. 

If $ \G = (B, \psi, A) $ is a directed quantum graph we shall say that an action $ \beta: B \rightarrow B \otimes C(G) $ 
of a compact quantum group $ G $ is compatible with $ A: B \rightarrow B $ if $ \beta \circ A = (A \otimes \id) \circ \beta $. 
Motivated by the considerations in \cite{Banicaqutgraph}, we define the quantum automorphism group of a directed quantum graph as follows, 
compare \cite{BCEHPSWbigalois}. 

\begin{definition}
Let $ \G = (B, \psi, A) $ be a directed quantum graph. The quantum automorphism group $ G^+(\G) $ of $ \G $ is the universal compact quantum 
group equipped with a $ \psi $-preserving action $ \beta: B \rightarrow B \otimes C(G^+(\G)) $ which is compatible with the quantum adjacency matrix $ A $. 
\end{definition} 

That is, if $ G $ is a compact quantum group and $ \gamma: B \rightarrow B \otimes C(G) $ an action of $ G $ which preserves $ \psi $ 
and is compatible with $ A $, then there exists a unique $ * $-homomorphism $ \pi: C(G^+(\G)) \rightarrow C(G) $, compatible with the comultiplications, 
such that the diagram 
$$
\xymatrix{
B \ar@{->}[r]^{\!\!\!\!\!\!\!\!\!\!\!\!\!\!\!\!\!\!\!\!\!\! \beta} \ar@{->}[rd]_{\gamma} & B \otimes C(G^+(\G)) \ar@{->}[d]^{\id \otimes \pi} \\ 
& B \otimes C(G) 
}
$$
is commutative. 

Comparing this with Definition \ref{defqut}, it is straightforward to check that $ C(G^+(\G)) $ can be identified with the quotient of $ C(G^+(B,\psi)) $ 
obtained by imposing the relation $ (1 \otimes A) v = v (1 \otimes A) $ on the fundamental matrix $ v $ of $ G^+(B,\psi) $. 

\begin{remark} 
If $ \G = K(M_N(\mathbb C), \tr) $ or $ \G = TM_N $ is the complete or the trivial quantum graph on $ M_N(\mathbb{C}) $, then it is easy to see that 
compatibility with the quantum adjacency matrix is in fact automatic. That is, we have $ G^+(\G) = G^+(M_N(\mathbb{C}), \tr) $ in either case. 
\end{remark} 

Let us recall that two compact quantum groups $ G_1, G_2 $ are called monoidally equivalent if their representation categories $ \Rep(G_1) $ and $ \Rep(G_2) $ 
are unitarily equivalent as $ C^* $-tensor categories \cite{BdRV}, \cite{NTlecturenotes}. A monoidal equivalence is a unitary tensor 
functor $ F: \Rep(G^+(\G_1)) \rightarrow \Rep(G^+(\G_2)) $ whose underlying functor is an equivalence.  

Assume that $ \G_i = (B_i, \psi_i, A_i) $ are directed quantum graphs for $ i = 1,2 $. Then the quantum automorphism group $ G^+(\G_i) $ is a 
quantum subgroup of $ G^+(B_i, \psi_i) $ such that the quantum adjacency matrix $ A_i $ is an intertwiner 
for the defining representation $ B_i = L^2(B_i) $ of $ G^+(\G_i) $. Note also that the multiplication map $ m_i: B_i \otimes B_i \rightarrow B_i $ 
and the unit map $ u_i: \mathbb{C} \rightarrow B_i $ are intertwiners for the action of $ G^+(\G_i) $, so that $ B_i $ 
becomes a monoid object in the tensor category $ \Rep(G^+(\G_i)) $. 

In analogy to \cite{BCEHPSWbigalois} we give the following definition. 

\begin{definition} \label{defqiso}
Two directed quantum graphs $ \G_i = (B_i, \psi_i, A_i) $ for $ i = 1,2 $ are quantum isomorphic if there exists a monoidal 
equivalence $ F: \Rep(G^+(\G_1)) \rightarrow \Rep(G^+(\G_2)) $ such that 
\begin{bnum} 
\item[a)] $ F $ maps the monoid object $ B_1 $ to the monoid object $ B_2 $. 
\item[b)] $ F(A_1) = A_2 $.  
\end{bnum} 
We will write $ \G_1 \cong_q \G_2 $ in this case. 
\end{definition} 

From Definition \ref{defqiso} it is easy to see that the notion of quantum isomorphism is an equivalence relation on isomorphism classes of directed 
quantum graphs. For concrete computations it is however more convenient to describe quantum isomorphisms in terms of bi-Galois 
objects \cite{BCEHPSWbigalois}, sometimes also called linking algebras.  

Concretely, if $ \G_i = (B_i, \psi_i, A_i) $ for $ i = 1,2 $ are directed quantum graphs then $ \O(G^+(\G_2, \G_1)) $ is the bi-Galois object 
generated by the coefficients of a unital $ \ast $-homomorphism 
$$
\beta_{ji}: B_i \to B_j \otimes \Poly(G^+(\G_j, \G_i)) 
$$ 
satisfying the conditions
$$
(\psi_j \otimes \id) \beta_{ji}(x) = \psi_i(x)1 
$$
for all $ x \in B_i $ and 
$$ 
(A_j \otimes \id) \beta_{ji} = \beta_{ji} A_i. 
$$ 
Note that these conditions generalize the requirements on the action of the quantum automorphism group of a quantum graph to be state-preserving and 
compatible with the quantum adjacency matrix, respectively. 

We write $ C(G^+(\G_j, \G_i)) $ for the universal enveloping $ C^\ast $-algebra of $ \Poly(G^+(\G_j, \G_i)) $. 
In exactly the same way as in \cite{BCEHPSWbigalois} one then arrives at the following characterization of quantum isomorphisms.  

\begin{theorem} \label{starbigaloischar}
Let $ \G_1, \G_2 $ be directed quantum graphs. Then the following conditions are equivalent. 
\begin{bnum} 
\item[a)] $ \G_1 $ and $ \G_2 $ are quantum isomorphic.  
\item[b)] $ \Poly(G^+(\G_2, \G_1)) $ is non-zero. 
\item[c)] $ \Poly(G^+(\G_2, \G_1)) $ admits a nonzero faithful state. 
\item[d)] $ C(G^+(\G_2, \G_1)) $ is non-zero. 
\end{bnum}
\end{theorem} 

If the equivalent conditions in Theorem \ref{starbigaloischar} are satisfied then $ \Poly(G^+(\G_2, \G_1)) $ is a 
$ \Poly(G^+(\G_2)) $-$ \Poly(G^+(\G_1)) $ bi-Galois object in a natural way \cite{Shopfgalois}. 
In particular, there exist ergodic left and right actions of $ G^+(\G_2) $ and $ G^+(\G_1) $ on $ \Poly(G^+(\G_2, \G_1)) $, respectively. 
Moreover, $ \Poly(G^+(\G_2, \G_1)) $ is equipped with a unique faithful state which is invariant with respect to both actions. 

For $ \G_1 = \G_2 $ and the identity monoidal equivalence, the $ \ast $-algebra $ \Poly(G^+(\G_2, \G_1)) $ equals $ \Poly(G^+(\G_1)) = \Poly(G^+(\G_2)) $, 
both actions are implemented by the comultiplication, and the invariant faithful state is nothing but the Haar state. 

\begin{remark} 
The abelianization of $ \Poly(G^+(\G_2, \G_1)) $ is the algebra of coordinate functions on the space of ``classical isomorphisms'' 
between the quantum graphs $ \G_1 $ and $ \G_2 $, that is, the space of unital $ \ast $-isomorphisms $ \varphi: B_1 \to B_2 $ satisfying 
$$
\psi_2 \circ \varphi = \psi_1, \qquad A_2 \circ \varphi = \varphi \circ A_1.
$$ 
If moreover each $ \G_i $ is associated with a classical directed graph $ E_i = (E_i^0, E_i^1) $ as in paragraph \ref{parclassical},  then by Gelfand duality 
such a map $ \varphi $ corresponds precisely to a graph isomorphism $ \varphi_*: E_2 \to E_1 $ via $ \varphi(f) = f \circ \varphi_* $ for $ f \in C(E_1^0) $. 
This is the reason for the ordering of the quantum graphs in our notation $ \Poly(G^+(\G_2, \G_1)) $.
\end{remark} 

\begin{remark}
There is a canonical algebra isomorphism $ S: \Poly(G^+(\G_2, \G_1)) \to \Poly(G^+(\G_1, \G_2))^{op} $, which can be viewed as a generalization of the 
antipode of the Hopf $ \ast $-algebra associated to a compact quantum group. 
More precisely, if $ (e_m) $ and $ (f_n) $ are orthonormal bases for $ B_1 $ and $ B_2 $, respectively, and we 
write $ \beta_{21}(e_m) = \sum_n f_n \otimes u_{nm} $, then $ u = (u_{ij}) \in \End(B_1, B_2) \otimes \Poly(G^+(\G_2, \G_1)) $ is a unitary matrix, 
and there is an algebra isomorphism $ S: \Poly(G^+(\G_2, \G_1)) \to \Poly(G^+(\G_1, \G_2))^{op} $ given by  
$$
(\id \otimes S)(u) = u^* = u^{-1}, \qquad (\id \otimes S)(u^*) = (J_2^t \bar{u} (J_1^{-1})^t),
$$
where $ (\bar{u})_{kl} = (u_{kl}^*), J_i: B_i \to B_i $ is the anti-linear involution map given by $ J_i(b) = b^* $ and $ t $ denotes
the transpose map. We refer to \cite{BCEHPSWbigalois} for more details. 
\end{remark} 

\begin{remark} \label{wealth}
We have a wealth of examples quantum isomorphisms between the complete quantum graphs $ K(B,\psi) $ introduced in paragraph \ref{parQKn}, and also between 
the trivial quantum graphs $ T(B,\psi) $ introduced in paragraph \ref{parQTn}. Recall that $ K(B, \psi) $ (resp. $ T(B,\psi) $) is defined by equipping the finite 
quantum space $ (B, \psi) $ with the quantum adjacency matrix $ A: L^2(B) \rightarrow L^2(B) $ given by $ A(b) = \delta^2 \psi(b) 1 $ (resp. $ A = \id $). 
More precisely, if $ (B_i, \psi_i) $ are finite quantum spaces for $ i = 1,2 $, with $ \delta_i $-forms $ \psi_i $, then 
$$
K(B_1, \psi_1) \cong_q K(B_2,\psi_2) \iff T(B_1, \psi_1) \cong_q T(B_2, \psi_2) \iff \delta_1 = \delta_2 .
$$ 
These equivalences follow from work of DeRijdt and Vander Vennet in \cite{dRV}, where unitary fiber functors on quantum automophism groups of finite quantum 
spaces equipped with $ \delta $-forms were classified. 
\end{remark} 

Let $ \G_i = (B_i, \psi_i, A_i) $ be directed quantum graphs in standard form, in the sense explained in paragraph \ref{parqgraph}. Explicitly, we fix 
multimatrix decompositions 
$$
B_i = \bigoplus_{a = 1}^{d_i} M_{N^i_a}(\mathbb{C})
$$ 
and diagonal positive invertible matrices $ Q^i_{(a)} $ implementing $ \psi_i $. Let us express the quantum adjacency matrices relative to 
the standard matrix units $ e_{kl}^{(a)} \in M_{N^i_a}(\mathbb{C}) $, so that 
$$
A_i(e^{(a)}_{kl}) = \sum_{brs} (A_i)_{kla}^{rsb} e^{(b)}_{rs}. 
$$
We then obtain the following result, compare \cite{Mrozinskiso3deformations} for the case $ \G_1 = \G_2 $. 

\begin{prop} \label{qisorelations}
Let $ \G_1 $ and $ \G_2 $ be directed quantum graphs given as above. Then $ \Poly(G^+(\G_2, \G_1)) $ is the universal unital $ \ast $-algebra with 
generators $ v_{ija}^{klb} $ for $ 1 \leq i,j \leq N^1_a $, $ 1 \leq k,l \leq N^2_b $, $ 1 \leq a \leq d_1 $, $ 1 \leq b \leq d_2 $, satisfying the 
relations
\begin{itemize}
\item[(A1a)] $ \sum_w v_{kla}^{xwc} v_{rsb}^{wyc} = \delta_{ab} \delta_{lr} v_{ksa}^{xyc} $
\item[(A1b)] $ \sum_l (Q^1_{(a)})^{-1}_{ll} v_{mla}^{xwb} v_{lka}^{zyc} = \delta_{bc} \delta_{wz} (Q^2_{(c)})^{-1}_{zz} v_{mka}^{xyc} $ 
\item[(A2)] $ (v_{kla}^{xyb})^* = v_{lka}^{yxb} $
\item[(A3a)] $ \sum_{xb} (Q^2_{(b)})_{xx} v^{xxb}_{kla} = \delta_{kl} (Q^1_{(a)})_{kk} $
\item[(A3b)] $ \sum_{ka} v^{xyb}_{kka} = \delta_{xy} $
\item[(A4)] $ \sum_{rsb} (A_2)^{xyc}_{rsb} v^{rsb}_{kla} = \sum_{rsb} (A_1)_{kla}^{rsb} v_{rsb}^{xyc} $
\end{itemize}
for all admissible indices. 
\end{prop}

\begin{proof}
The following argument is analogous to the one for Proposition 2.10 in \cite{Mrozinskiso3deformations}, compare \cite{Wangqsymmetry}. 
Expressing the universal morphism $ \beta_{21}: B_1 \to B_2 \otimes \Poly(G^+(\G_2, \G_1))$ relative to the bases chosen as above, we can write 
$$ 
\beta_{21}(e^{(a)}_{kl}) = \sum_{xyb} e_{xy}^{(b)} \otimes v^{xyb}_{kla}.
$$
Then $ \Poly(G^+(\G_2, \G_1)) $ is generated as a $ \ast $-algebra by the elements $ v^{xyb}_{kla} $. Now the conditions on this implementing 
a bi-Galois object are equivalent to the equations listed above. More precisely, we have 
\begin{bnum}
\item[$\bullet$] \emph{$ (A1a) \iff \beta_{21} $ is an algebra homomorphism}. 
This follows from 
$$
\beta_{21}(e^{(a)}_{kl}) \beta_{21}(e^{(b)}_{rs}) = \sum_{xwc myd} e_{xw}^{(c)} e_{my}^{(d)} \otimes v^{xwc}_{kla} v^{myd}_{rsb} 
= \sum_{xwc n} e_{xy}^{(c)} \otimes v^{xwc}_{kla} v^{wyc}_{rsb} 
$$
and 
$$
\beta_{21}(e^{(a)}_{kl} e^{(b)}_{rs}) = \delta_{ab} \delta_{lr} \beta_{ji}(e^{(a)}_{ks}) = \sum_{xyc} \delta_{ab} \delta_{lr} e_{xy}^{(c)} \otimes v^{xyc}_{ksa}.
$$
\item[$\bullet$] \emph{$(A1b) \iff S: \Poly(G^+(\G_1, \G_2)) \to \Poly(G^+(\G_2, \G_1)) $ given by 
$$ 
S(v^{kla}_{rsb}) = (Q^2_{(b)})_{ss} (Q^1_{(a)})^{-1}_{ll} v_{lka}^{srb} 
$$ 
defines an algebra anti-isomorphism}. Indeed, we have 
$$
\sum_l S(v^{lma}_{wxb}) S(v^{kla}_{yzc}) = \sum_l (Q^2_{(b)})_{xx} (Q^1_{(a)})^{-1}_{mm} v_{mla}^{xwb} (Q^2_{(c)})_{zz} (Q^1_{(a)})^{-1}_{ll} v_{lka}^{zyc} 
$$
and 
$$
\delta_{bc} \delta_{wz} S(v^{kma}_{yxc}) = \delta_{bc} \delta_{wz} (Q^2_{(c)})_{xx} (Q^1_{(a)})^{-1}_{mm} v_{mka}^{xyc},  
$$
so this statement follows in combination with $ (A1a) $. 
\item[$\bullet$] \emph{$(A2) \iff \beta_{21} $ is involutive}. This follows immediately from $ (e^{(a)}_{kl})^* = e^{(a)}_{lk} $.
\item[$\bullet$] \emph{$(A3a) \iff (\psi_2 \otimes \id) \circ \beta_{21}(b) = \psi_1(b)1 $ for all $ b \in B_1 $}.
This follows from 
$$
(\psi_2 \otimes \id) \circ \beta_{21}(e^{(a)}_{kl}) = \sum_{xyb} \psi_2(e_{xy}^{(b)}) v^{xyb}_{kla} = \sum_{xb} (Q_{(b)}^2)_{xx} v^{xxb}_{kla}
$$
and 
$$
\psi_1(e^{(a)}_{kl})1 = (Q_{(a)}^1)_{kk} \delta_{kl}. 
$$
\item[$\bullet$] \emph{$(A3b) \iff \beta_{21} $ is unital}. This follows from 
$$
\beta_{21}(1) = \sum_{ak} \beta_{21}(e^{(a)}_{kk}) = \sum_{xybak} e_{xy}^{(b)} \otimes v^{xyb}_{kka}. 
$$
\item[$\bullet$] \emph{$(A4) \iff \beta_{21} \circ A_1 = (A_2 \otimes \id) \circ \beta_{21} $}. This follows from 
$$
(\beta_{21} \circ A_1)(e^{(a)}_{kl}) = \sum_{rsb} (A_1)_{kla}^{rsb} \beta_{21}(e^{(b)}_{rs}) = \sum_{rsbxyc} (A_1)_{kla}^{rsb} e^{(c)}_{xy} \otimes v_{rsb}^{xyc}
$$
and 
$$
(A_2 \otimes \id) \circ \beta_{21}(e^{(a)}_{kl}) = \sum_{rsb} A_2(e_{rs}^{(b)}) \otimes v^{rsb}_{kla} 
= \sum_{rsbxyc} (A_2)_{rsb}^{xyc} e^{(c)}_{xy} \otimes v^{rsb}_{kla}. 
$$
\end{bnum} 
Combining these observations yields the claim. 
\end{proof}

\subsection{Quantum symmetries of quantum Cuntz-Krieger algebras} 
 
We shall now show that quantum automorphisms and quantum isomorphisms of directed quantum graphs lift naturally to the level of 
their associated $ C^\ast $-algebras. 

Firstly, we have the following lifting result for quantum symmetries, compare the work in \cite{SchmidtWeberqsym} on classical graph $ C^\ast $-algebras.  

\begin{theorem} \label{qutprop} 
Let $ \G = (B, \psi, A) $ be a directed quantum graph. Then the canonical action $ \beta: B \rightarrow B \otimes C(G^+(\G)) $ 
of the quantum automorphism group of $ \G $ induces an action $ \hat{\beta}: \FO(\G) \rightarrow \FO(\G) \otimes C(G^+(\G)) $ such that
$$
\hat{\beta}(S(b)) = (S \otimes \id)\beta(b) 
$$
for all $ b \in B $. 
\end{theorem} 

The proof of Theorem \ref{qutprop} will be obtained as a special case of the more general Theorem \ref{qisolift} on quantum isomorphisms below. 
Nonetheless, for the sake of clarity we have decided to state this important special case separately. 

\begin{remark} 
There are typically plenty of quantum automorphisms of $ \FO(\G) $, and in fact, even $ * $-automorphisms, which do not arise from quantum automorphisms 
as in Theorem \ref{qutprop}. For instance, the gauge action on the free Cuntz-Krieger algebra associated with a classical directed graph cannot be described 
this way, compare paragraph \ref{pargauge}.   
\end{remark} 

Now assume that $ \G_1, \G_2 $ are quantum isomorphic directed quantum graphs in standard form, with corresponding linking algebras $ \Poly(G^+(\G_j, \G_i)) $.  
The associated $ \ast $-homomorphisms $ \beta_{ji}: B_i \to B_j \otimes \Poly(G^+(\G_j,\G_i)) $ for $ 1 \leq i, j \leq 2 $ are given by 
$$
\beta_{ji}(e^{(a)}_{kl}) = \sum_{xyb} e_{xy}^{(b)} \otimes v^{xyb}_{kla}
$$ 
in terms of the standard matrix units. Here $ v^{xyb}_{kla} $ are the generators of $ \Poly(G^+(\G_j,\G_i)) $ as in Proposition \ref{qisorelations}. 

\begin{theorem} \label{qisolift}
Let $ \G_i = (B_i, \psi_i, A_i) $ for $ i = 1,2 $ be directed quantum graphs and assume that $ \G_1 \cong_q \G_2 $. Then there exists $ \ast $-homomorphisms 
$$
\hat{\beta}_{ji}: \FO(\G_i) \to \FO(\G_j) \otimes C(G^+(\G_j, \G_i))
$$
for $ 1 \leq i,j \leq 2 $ such that 
$$ 
\hat{\beta}_{ji}(S_i(b)) = (S_j \otimes \id) \beta_{ji}(b) 
$$ 
for all $ b \in B_i $.
\end{theorem}

\begin{proof}
Observe first that for $ i = j $ we are precisely in the situation of Theorem \ref{qutprop}, so that Theorem \ref{qutprop} is indeed 
a special case of the claim at hand.  

Let us write $ m_\O: \O \otimes \O \rightarrow \O $ for the multiplication in $ \O = \Poly(G^+(\G_j, \G_i)) $. We claim that
$$ 
(m_j^* \otimes \id)\beta_{ji} = (\id \otimes \id \otimes m_\O)(\id \otimes \sigma \otimes \id)(\beta_{ji} \otimes \beta_{ji}) m_i^*,
$$ 
where $ m_i: B_i \rightarrow B_i \rightarrow B_i $ denotes multiplication in $ B_i $ and $ \sigma $ is the flip map. 
Indeed, rewriting Lemma \ref{mstarcomputation} in terms of the standard matrix units yields
$$
m_i^*(e^{(a)}_{kl}) = \sum_r (Q^i_{(a)})^{-1}_{rr}\, e^{(a)}_{kr} \otimes e^{(a)}_{rl},
$$ 
and using relation $ (A1b) $ from Proposition \ref{qisorelations} we get 
\begin{align*}
(m_j^* \otimes \id)\beta_{ji}(e^{(a)}_{kl}) &= \sum_{xnb} m_j^*(e^{(b)}_{xn}) \otimes v^{xnb}_{kla} \\
&= \sum_{xybn} (Q_{(b)}^j)^{-1}_{yy} e^{(b)}_{xy} \otimes e^{(b)}_{yn} \otimes v^{xnb}_{kla} \\
&= \sum_{xybmnc} (Q_{(b)}^j)^{-1}_{yy} \delta_{bc} \delta_{ym} e^{(b)}_{xy} \otimes e^{(c)}_{mn} \otimes v^{xnb}_{kla} \\
&= \sum_{wxybmnc} (Q_{(a)}^i)^{-1}_{ww} e^{(b)}_{xy} \otimes e^{(c)}_{mn} \otimes v^{xyb}_{kwa} v^{mnc}_{wla} \\
&= \sum_w (Q_{(a)}^i)^{-1}_{ww} (\id \otimes \id \otimes m_\O)(\id \otimes \sigma \otimes \id)(\beta_{ji} \otimes \beta_{ji}) (e^{(a)}_{kw} \otimes e^{(a)}_{wl}) \\
&= (\id \otimes \id \otimes m_\O)(\id \otimes \sigma \otimes \id)(\beta_{ji} \otimes \beta_{ji}) m_i^*(e^{(a)}_{kl})  
\end{align*}
as required. 

Now consider the linear map $ s: B_i \rightarrow \FO(\G_j) \otimes C(G^+(\G_j, \G_i)) = D $ given by $ s = (S_j \otimes \id) \beta_{ji} $. Then 
$$ 
s^*(b) = s(b^*)^* = (S_j \otimes \id)\beta_{ji}(b^*)^* = (S_j^* \otimes \id)\beta_{ji}(b), 
$$
and we claim that $ s $ is a quantum Cuntz-Krieger $ \G_i $-family in $ D $. Writing $ \mu $ for the multiplication in $ \FO(\G_j) $ 
and $ \mu_D $ for the one in $ D $, our above considerations yield
\begin{align*}
&\mu_D (\id \otimes \mu_D)(s \otimes s^* \otimes s)(\id \otimes m_i^*)m_i^* \\
&= \mu_D (\id \otimes \mu_D)(S_j \otimes \id \otimes S_j^* \otimes \id \otimes S_j \otimes \id)(\beta_{ji} \otimes \beta_{ji} \otimes \beta_{ji})
(\id \otimes m_i^*)m_i^* \\
&= \mu_D (\id \otimes \id \otimes \mu \otimes \id)(S_j \otimes \id \otimes S_j^* \otimes S_j \otimes m_\O) \sigma_{45} 
(\beta_{ji} \otimes \beta_{ji} \otimes \beta_{ji}) (\id \otimes m_i^*)m_i^* \\
&= \mu_D(\id \otimes \id \otimes \mu \otimes \id)(S_j \otimes \id \otimes S_j^* \otimes S_j \otimes \id)
(\id \otimes \id \otimes m_j^* \otimes \id)(\beta_{ji} \otimes \beta_{ji})m_i^* \\
&= (\mu \otimes \id)(\id \otimes \mu \otimes \id)(S_j \otimes S_j^* \otimes S_j \otimes m_\O)
(\id \otimes m_j^* \otimes \id)(\id \otimes \sigma \otimes \id)(\beta_{ji} \otimes \beta_{ji})m_i^* \\
&= (\mu \otimes \id)(\id \otimes \mu \otimes \id)(S_j \otimes S_j^* \otimes S_j \otimes \id)(\id \otimes m_j^* \otimes \id) (m_j^* \otimes \id)\beta_{ji} \\
&= (S_j \otimes \id)\beta_{ji} = s,  
\end{align*}
and similarly
\begin{align*}
\mu_D (s^* \otimes s) m_i^* 
&= (\mu \otimes m_\O) \sigma_{23} (S_j^* \otimes \id \otimes S_j \otimes \id)(\beta_{ji} \otimes \beta_{ji})m_i^* \\ 
&= (\mu \otimes \id)(S_j^* \otimes S_j \otimes m_\O) \sigma_{23} (\beta_{ji} \otimes \beta_{ji})m_i^* \\ 
&= (\mu \otimes \id)(S_j^* \otimes S_j \otimes \id)(m_j^* \otimes \id)\beta_{ji} \\ 
&= (\mu \otimes \id)(S_j \otimes S_j^* \otimes \id)(m_j^* \otimes \id)(A_j \otimes \id)\beta_{ji} \\ 
&= (\mu \otimes \id)(S_j \otimes S_j^* \otimes \id)(m_j^* \otimes \id)\beta_{ji} A_i \\ 
&= \mu_D (s \otimes s^*) m_i^* A_i, 
\end{align*}
using the quantum Cuntz-Krieger relation for $ S_j $. Hence the universal property of $ \FO(\G_i) $ yields the claim. 
\end{proof}
 
\begin{remark} \label{remhatinjective} 
If we denote by $ \C_i \subset \FO(\G_i) $ the dense $ \ast $-subalgebra generated by $ S_i(B_i) $, then the restriction of the map $ \hat{\beta}_{ji} $ 
in Theorem \ref{qisolift} to $ \C_i $ is injective. Indeed, there exists a canonical unital $ \ast $-isomorphism 
\begin{align*}
\theta_i^j: \Poly(G^+(\G_i)) &\to \Poly(G^+(\G_i, \G_j)) \Box_{\Poly(\G_j)} \Poly(G^+(\G_j, \G_i)),  
\end{align*}
where 
\begin{align*}
\Poly(G^+(\G_i, \G_j)) \Box_{\Poly(\G_j)} \Poly(G^+(\G_j, \G_i)) &= \{x \mid (\rho_j \otimes \id)(x) = (\id \otimes \lambda_j)(x) \} \\
&\subset \Poly(G^+(\G_i, \G_j)) \otimes \Poly(G^+(\G_j, \G_i)), 
\end{align*}
and 
\begin{align*}
\rho_j&: \Poly(G^+(\G_i, \G_j)) \rightarrow \Poly(G^+(\G_i, \G_j)) \otimes \Poly(G^+(\G_j)) \\
\lambda_j&: \Poly(G^+(\G_j, \G_i)) \rightarrow \Poly(G^+(\G_j)) \otimes \Poly(G^+(\G_j, \G_i)) 
\end{align*}
are the canonical ergodic actions of $ G^+(\G_j) $ on the linking algebras. 
The map $ \theta_i^j $ satisfies 
$$
(\hat{\beta}_{ij} \otimes \id) \hat{\beta}_{ji}(x) = (\id \otimes \theta_i^j) \hat{\beta}_{ii}(x) 
$$
for all $ x \in \C_i $. 
If $ \epsilon_i: \Poly(G^+(\G_i, \G_j)) \Box_{\Poly(\G_j)} \Poly(G^+(\G_j, \G_i)) \cong \Poly(G^+(\G_i)) \to \mathbb{C} $ is the character given by 
the counit of $ \Poly(G^+(\G_i)) $, this implies 
$$
(\id \otimes \epsilon_i)(\hat{\beta}_{ij} \otimes \id) \hat{\beta}_{ji}(x) = x 
$$
for $ x \in \C_i $. Hence the restriction of $ \hat{\beta}_{ji} $ to $ \C_i $ is indeed injective.

However, it is not clear whether the map $ \hat{\beta}_{ji}: \FO(\G_i) \to \FO(\G_j) \otimes C(G^+(\G_j, \G_i)) $ itself is injective. In 
the following section we show that this is at least sometimes the case.
\end{remark}

\section{Unitary error bases and finite dimensional quantum symmetries} \label{secunitaryerror} 

In this section we apply the general results of the previous section to certain pairs of complete quantum graphs and 
trivial quantum graphs, respectively. More precisely, we fix $ N \in \mathbb{N} $ and consider 
\begin{align*}
\G_1^K(N) &= K_{N^2} = K(\mathbb{C}^{N^2}, \tr) \\ 
\G_2^K(N) &= K(M_N(\mathbb{C}), \tr) 
\end{align*}
and
\begin{align*}
\G_1^T(N) &= T_{N^2} = T(\mathbb{C}^{N^2}, \tr) \\ 
\G_2^T(N) &= T(M_N(\mathbb{C}), \tr) = TM_N,  
\end{align*}
compare section \ref{secexamples}. 
The similarity between these pairs stems from the fact that we have canonical identifications 
\begin{align*}
G^+(\G_1^K(N)) &= G^+(\G_1^T(N)) = S_{N^2}^+, \\
G^+(\G_2^K(N)) &= G^+(\G_2^T(N)) = G^+(M_N(\mathbb{C}), \tr), 
\end{align*}
respectively. We will therefore also use the short hand notation $ G^+(\G_1(N)) $ and $ G^+(\G_2(N)) $ for these quantum automorphism groups. 

We recall that $ G^+(\G_1(N)) $ and $ G^+(\G_2(N)) $ are monoidally equivalent, and that we have quantum isomorphisms $ \G_1^K(N) \cong_q \G_2^K(N) $ 
and $ \G_1^T(N) \cong_q \G_2^T(N) $, see the remarks in paragraph \ref{parqsym}. This means in particular that there exists a 
bi-Galois object $ \Poly(G^+(\G_2(N), \G_1(N))) $ linking $ G^+(\G_1(N)) $ and $ G^+(\G_2(N)) $. If $ X $ is a set of cardinality $ N^2 $, then 
this $ \ast $-algebra can be described in terms of generators $ v^{rs}_x $ with $ 1 \leq r,s \leq N $ and $ x \in X $, satisfying the relations as 
in Proposition \ref{qisorelations}.

\subsection{Representations from unitary error bases}  

With some inspiration from quantum information theory, we shall now construct unital $ * $-homomorphisms from the linking algebra $ \Poly(G^+(\G_2(N), \G_1(N))) $ 
to $ M_N(\mathbb{C}) $. The key tool in this construction is the notion of a unitary error basis \cite{Wernerteleportation}. 

\begin{definition}
Let $ N \in \mathbb{N} $ and let $ X $ be a finite set of cardinality $ N^2 $. A {\it unitary error basis} for $ M_N(\mathbb{C}) $ is a 
basis $ \mathcal{W} = \{w_x\}_{x \in X} $ for $ M_N(\mathbb{C}) $ consisting of unitary matrices that are orthonormal with respect to the normalized 
trace inner product, so that 
$$
\tr(w_x^*w_y) = \delta_{xy} 
$$
for all $ x, y \in X $. 
\end{definition}

Unitary error bases play a fundamental role in quantum information theory. In particular, they form a one-to-one correspondence with ``tight'' quantum 
teleportation and superdense coding schemes \cite{Wernerteleportation}. 

\begin{prop} \label{constructpiW}
Let $ N \in \mathbb{N} $ and assume that $ \mathcal{W} = \{w_x\}_{x \in X} $ is a unitary error basis for $ M_N(\mathbb{C}) $. With the notation as above, 
there exists a unital $ \ast $-representation $ \pi_{\mathcal{W}}: \Poly(G^+(\G_2(N),\G_1(N))) \to M_N(\mathbb{C}) $ such that 
$$
\pi_{\mathcal{W}}(v^{rs}_x) = \frac{1}{N} w_x^* e_{rs} w_x
$$
for all $ r,s,x $. 
\end{prop}

\begin{proof}
Recalling that we write $ e_{rs} \in M_N(\mathbb{C}) $ for the standard matrix units, let us define
$$ 
V^{rs}_x = \frac{1}{N} w_x^* e_{rs} w_x 
$$ 
for all $ 1 \leq r,s \leq N $ and $ x \in X $. It suffices to check that the elements $ V^{rs}_x \in M_N(\mathbb{C}) $ satisfy  
the relations in Proposition \ref{qisorelations}. 

In order to do this, we recall from Theorem 1 in \cite{Wernerteleportation} that a unitary error basis 
can be equivalently characterized by the following properties for a family of unitaries $ \mathcal{W} = \{w_x\}_{x \in X} \subset M_N(\mathbb{C}) $:

\begin{bnum}
\item[a)] ({\it Depolarizing property}): $ \sum_{x \in X} w_x^*a w_x = N \Tr(a) 1 $ for $ a \in M_N(\mathbb{C}) $.
\item[b)] ({\it Maximally entangled basis property}): If $ \Omega = \frac{1}{\sqrt{N}}\sum_{i = 1}^N e_{i} \otimes e_{i} \in \mathbb{C}^N \otimes \mathbb{C}^N $ 
is a maximally entangled state and $ \Omega_x = (w_x \otimes 1)\Omega $, then $ \{\Omega_x\}_{x \in X} $ is an orthonormal basis 
for $ \mathbb{C}^N \otimes \mathbb{C}^N $.
\end{bnum} 

Observing that $ Q^1 = N^{-2} \id $ and $ Q_2 = N^{-1} \id $ we therefore we have to verify the following relations: 

\begin{bnum} 
\item[$\bullet$] \emph{$(A1a) \iff \sum_w V_{x}^{rw} V_{y}^{ws} = \delta_{xy} V_{x}^{rs} $}. This follows from 
\begin{align*}
\sum_t V_{x}^{rt} V_{y}^{ts} &= N^{-2} \sum_t w_x^*e_{rt} w_x w_y^*e_{ts}w_y \\
&= N^{-2} \Tr(w_xw_y^*) w_x^*(e_{rs})w_y = \delta_{xy} N^{-1} w_x^*(e_{rs})w_x = \delta_{xy} V_{x}^{rs}.
\end{align*}

\item[$\bullet$] \emph{$ (A1b) \iff V^{ji}_{x} V^{sr}_{x} = \delta_{is} N^{-1} V^{jr}_{x} $}. 
This follows directly from 
$$ 
(w_x^*e_{ji}w_x)(w_x^*e_{sr}w_x) = \delta_{is} w_x^*e_{jr}w_x.
$$

\item[$\bullet$] \emph{$ (A2) \iff (V^{ij}_{x})^* = V^{ji}_{x} $}. This is immediate.

\item[$\bullet$] \emph{$ (A3a) \iff \sum_i N V^{ii}_{x} = 1 $}. This follows from 
$$
\sum_i N V^{ii}_{x} = \sum_i w_x^* e_{ii} w_x = w_x^* w_x = 1.
$$

\item[$\bullet$] \emph{$ (A3b) \iff \sum_{z} V^{ij}_{z} = \delta_{ij}1 $}. This is the depolarizing property of $ \mathcal{W} $. 

\item[$\bullet$] \emph{$ (A4) \iff \sum_{rs} (A_2)_{rs}^{ij} V^{rs}_{x} = \sum_y (A_1)^{y}_{x} V_{y}^{ij} $}. 
For the trivial quantum graphs this is obvious. In the case of complete quantum graphs we have $ (A_1)^x_y = 1, (A_2)^{ij}_{kl} = N \delta_{ij} \delta_{kl} $
for all $ x,y,i,j,k,l $. Combining this with relations (A3a) and (A3b) yields the claim. More precisely, using (A3a) we obtain
\begin{align*}
\text{(A4)} &\iff \delta_{ij}\sum_s N V^{ss}_x = \sum_{y} V_{y}^{ij} \\
&\iff \delta_{ij}1 = \sum_y V^{ij}_y \\
&\iff \text{(A3b)}
\end{align*}
as required.
\end{bnum} 
This completes the proof. 
\end{proof}

\begin{remark} \label{pauli}
It is easy to construct examples of unitary error bases. 
Let $ X, Z \in M_N(\mathbb{C}) $ be the generalized Pauli matrices given by their action on the standard basis $ |0 \ket, \dots, |N - 1 \ket $ 
of $ \mathbb{C}^N $ according to the formulas
$$
X|j \ket  = \omega^j |j \ket, \qquad Z|j \ket  = |j + 1 \ket, 
$$     
where we write $ \omega = e^{\frac{2\pi i}{N}} $ and calculate modulo $ N $. Then $ \mathcal{W} = \{X^jZ^k\}_{0 \leq j,k \leq N-1} $ is a unitary error 
basis for $ M_N(\mathbb{C}) $. 
\end{remark}

\subsection{Applications} 

We shall use Proposition \ref{constructpiW} to study the structure of the quantum Cuntz-Krieger algebras $ \FO(\G_2^K(N)) = \FO(K(M_N,\tr)) $ 
and $ \FO(\G_2^T(N)) = \FO(TM_N) $, by comparing them with $ \FO(\G_1^K(N)) $ and $ \FO(\G_1^T(N)) $, respectively. 
Recall from paragraph \ref{parQKn} that $ \FO(\G_1^K(N)) $ identifies canonically with the Cuntz algebra $ \O_{N^2} $. 
Note also that $ \FO(\G_1^T(N)) = \ast_{N^2} C(S^1) $ is the non-unital free product of $ N^2 $ copies of $ C(S^1) $, compare Proposition \ref{sumfreeproduct}.  

\begin{prop} \label{Cuntzembedding}
There are injective $ \ast $-homomorphisms
\begin{align*}
\pi_N^K&: \O_{N^2} \hookrightarrow M_N(\FO(K(M_N, \tr))) \\
\sigma_N^K&: \FO(K(M_N, \tr)) \hookrightarrow M_N(\O_{N^2})
\end{align*}
and 
\begin{align*}
\pi_N^T&: \ast_{N^2} C(S^1) \hookrightarrow M_N(\FO(TM_N)) \\
\sigma_N^T&: \FO(TM_N) \hookrightarrow M_N(\ast_{N^2}C(S^1))
\end{align*}
for all $ N \in \mathbb{N} $. 
\end{prop}

\begin{proof}
The construction of these maps for trivial quantum graphs is virtually identical to the one for complete quantum graphs. In order to treat both cases 
simultaneously we will therefore write $ \G_1(N) $ and $ \G_2(N) $ to denote either $ \G_1^K(N) $ and $ \G_2^K(N) $, or $ \G_1^T(N) $ 
and $ \G_2^T(N) $, respectively. Our task is then to define injective $ \ast $-homomorphisms
\begin{align*}
\pi_N&: \FO(\G_1(N)) \hookrightarrow M_N(\FO(\G_2(N))) \\
\sigma_N&: \FO(\G_2(N)) \hookrightarrow M_N(\FO(\G_1(N)))
\end{align*}
for $ N \in \mathbb{N} $. 

Since the quantum graphs $ \G_1(N) $ and $ \G_2(N) $ are quantum isomorphic, Theorem \ref{qisolift} yields natural $ \ast $-homomorphisms
\begin{align*}
&\hat{\beta}: \FO(\G_1(N)) \to \FO(\G_2(N)) \otimes C(G^+(\G_2(N), \G_1(N))) \\
&\hat{\gamma}: \FO(\G_2(N)) \to \FO(\G_1(N)) \otimes C(G^+(\G_2(N), \G_1(N)))^{op}, 
\end{align*} 
taking into account that $ C(G^+(\G_1(N), \G_2(N))) \cong C(G^+(\G_2(N), \G_1(N)))^{op} $. 

In order to give explicit formulas for these maps let $ X $ be a set of cardinality $ N^2 $ and denote the standard generators 
of $ \FO(\G_1(N)) $ by $ S_x = S_1(N^2 e_x) $ for $ x \in X $. Similarly, write $ S_{rs} = S_2(N e_{rs}) $ for the standard generators 
of $ \FO(\G_2(N)) = \FO(\G_2) $. Here $ S_1: \mathbb{C}^{N^2} \rightarrow \FO(\G_1(N)) $ and $ S_2: M_N(\mathbb{C}) \rightarrow \FO(\G_2(N)) $ are the 
canonical linear maps. 
Then we calculate 
\begin{align*}
\hat{\beta}(S_x) &= N \sum_{rs} S_{rs} \otimes v^{rs}_x, \\
\hat{\gamma}(S_{rs}) &= \sum_x S_x \otimes v^{sr}_x, 
\end{align*}
where $ v^{rs}_x $ for $ 1 \leq r,s \leq N, x \in X $ are the standard generators of the linking algebra $ \O(G^+(\G_2(N), \G_1(N))) $, 
see Proposition \ref{qisorelations}. 

Consider now the $ \ast $-homomorphism $ \pi_\mathcal{W}: \Poly(G^+(\G_2(N),\G_1(N))) \to M_N(\mathbb{C}) $ obtained in Proposition \ref{constructpiW}. 
Slicing the maps $ \hat{\beta}, \hat{\gamma} $ with $ \pi_\mathcal{W} $, we define the desired $ \ast $-homomorphims $ \pi_N, \sigma_N $ by 
\begin{align*}
\pi_N &= (\id \otimes \pi_{\mathcal W}) \hat{\beta}, \\
\sigma_N &= (\id \otimes t) (\id \otimes \pi_{\mathcal W}^{op}) \hat{\gamma}, 
\end{align*}
using the isomorphism $ t: M_N(\mathbb{C})^{op} \cong M_N(\mathbb{C}) $ given by sending a matrix $ Y $ to its transpose $ Y^t $. 
Concretely, if we let $ V^{rs}_x $ be constructed out of a unitary error basis as in Proposition \ref{constructpiW} then we have 
\begin{align*}
\pi_N(S_x) &= N \sum_{rs} S_{rs} \otimes V^{rs}_x, \\
\sigma_N(S_{rs}) &= \sum_x S_x \otimes (V^{sr}_x)^t.  
\end{align*}
Let us denote by $ m: M_N(\mathbb{C})^{op} \otimes M_N(\mathbb{C}) \to M_N(\mathbb{C}), m(a^{op} \otimes b) = ab $ the multiplication map. 
Using the relations in Proposition \ref{constructpiW} we readily see that 
$$
(\id \otimes m)(\id \otimes t \otimes \id)(\sigma_N \otimes \id)\pi_N(a) = a \otimes 1 
$$
for all $ a $ contained in the $ \ast $-algebra generated by the elements $ S_x $ for $ x \in X $.
Similarly, we have 
$$
(\id \otimes m)(\id \otimes t \otimes \id)(\pi_N \otimes \id)\sigma_N(b) = b \otimes 1 
$$
for all $ b $ contained in the $ \ast $-algebra generated by the elements $ S_{rs} $. 
Since $ m $ is completely bounded, it follows by continuity that $ \pi_N $ and $ \sigma_N $ are injective. 
\end{proof}

Let us continue to use the notation from above and denote the embeddings obtained in Proposition \ref{Cuntzembedding} by $ \pi_N $ and $ \sigma_N $, 
referring to either the trivial or complete quantum graphs $ \G_1(N), \G_2(N) $. Following ideas of Gao, Harris and Junge \cite{GHJteleportation}, we shall 
refine these embeddings and realize each of $ \FO(\G_1(N)) $ and $ \FO(\G_2(N)) $ as an iterated crossed product of the other algebra with respect to 
certain $ \mathbb{Z}_N $-actions, up to tensoring with matrices. This is indeed very much related to the work in \cite{GHJteleportation}, which exhibited a 
similar connection between free group $ C^\ast $-algebras and Brown's universal non-commutative unitary algebras.

In the sequel we write again $ \omega = e^{\frac{2 \pi i}{N}} $ and calculate modulo $ N $. We relabel the generators of $ \FO(\G_i(N)) $ 
in the proof of Proposition \ref{Cuntzembedding} by $ S^{(i)}_{kl} $ for $ i = 1,2 $, with indices $ 0 \leq k,l \leq N-1 $, and let $ X, Z $ be the 
generalized Pauli matrices from Remark \ref{pauli}.  

Let us consider the following order $ N $ automorphisms $ \alpha_j \in \Aut(\FO(\G_2(N))) $ for $ j = 1,2 $ given on generators by
$$
\alpha_1(S^{(2)}_{kl}) = \omega^{k-l} S^{(2)}_{kl} \qquad \alpha_2(S^{(2)}_{kl}) = S^{(2)}_{k-1, l-1}. 
$$
Note that $ \alpha_1 $ and $ \alpha_2 $ can be viewed as examples of gauge automorphisms as in paragraph \ref{pargauge}. 
More precisely, they are the gauge automorphisms associated with the unitaries $ X, Z $ in the sense that  
$$
(\alpha_1 \otimes \id)(S) = (1 \otimes X) S (1 \otimes X^*) \qquad (\alpha_2 \otimes \id)(S) = (1 \otimes Z^*) S (1 \otimes Z)
$$
for $ S = (S^{(2)}_{ij}) \in \FO(\G_2(N)) \otimes M_N(\mathbb{C}) $. 
Similarly, we define order $ N $ automorphisms $ \beta_j \in \Aut(\FO(\G_1(N))) $ for $ j = 1,2 $ by 
$$
\beta_1(S^{(1)}_{kl}) = S^{(1)}_{k-1,l} \qquad \beta_2(S^{(1)}_{kl}) = S^{(1)}_{k, l-1}. 
$$ 
Clearly, all these automorphisms define actions of $ \mathbb{Z}_N $ on $ \FO(\G_2(N)) $ and $ \FO(\G_1(N)) $, respectively. 
From the relation $ XZ = \omega ZX $ it follows that both pairs of actions $ \alpha_1, \alpha_2 $ and $ \beta_1, \beta_2 $ mutually commute. 

Let us now consider the iterated crossed products 
$$
\FO(\G_2(N)) \rtimes_{\alpha_1} \mathbb{Z}_N \rtimes_{\alpha_2} \mathbb{Z}_N, \qquad \FO(\G_1(N)) \rtimes_{\beta_1} \mathbb{Z}_N \rtimes_{\beta_2} \mathbb{Z}_N, 
$$
where $ \alpha_2, \beta_2 $ are naturally extended to the crossed products by letting $ \mathbb{Z}_N $ act on itself through appropriate dual actions. 
More precisely, given $ a \in \FO(\G_2(N)), b \in \FO(\G_1(N)) $ and $ g \in \mathbb{Z}_N \cong \{0, \ldots, N-1\} $, we let 
$$
\alpha_2(a u_g) = \alpha_2(a) \omega^g u_g, \qquad \beta_2(b u_g) = \beta_2(g) \omega^{g} u_g. 
$$ 
Abstractly, the algebra $ \FO(\G_2(N)) \rtimes_{\alpha_1} \mathbb{Z}_N \rtimes_{\alpha_2} \mathbb{Z}_N $ is the universal $ C^\ast $-algebra spanned by 
elements of the form
$$
x = \sum_{j,k = 0}^{N - 1} a_{jk} v^j w^k,  
$$
where $ a_{jk} \in \FO(\G_2(N)) $ and $ v, w $ are unitaries, such that the relations
$$
v^N = w^N = 1, \qquad v a_{jk} = \alpha_1(a_{jk}) v, \qquad w a_{jk} = \alpha_2(a_{jk}) w = a_{jk}, \qquad wv = \omega vw
$$
are satisfied. A similar description holds for $ \FO(\G_1(N)) \rtimes_{\beta_1} \mathbb{Z}_N \rtimes_{\beta_2} \mathbb{Z}_N $.

Our aim is to establish the following description of the iterated crossed products obtained in this way. 

\begin{theorem} \label{crossedproducts}
For the double crossed products with respect to the actions of $ \mathbb{Z}_N $ introduced above one obtains $ \ast $-isomorphisms 
\begin{align*}
M_N(\FO(K(M_N(\mathbb{C}), \tr))) &\cong \O_{N^2} \rtimes_{\beta_1} \mathbb{Z}_N \rtimes_{\beta_2} \mathbb{Z}_N \\
M_N(\O_{N^2}) &\cong \FO(K(M_N(\mathbb{C}), \tr)) \rtimes_{\alpha_1} \mathbb{Z}_N \rtimes_{\alpha_2} \mathbb{Z}_N 
\end{align*}
and 
\begin{align*}
M_N(\FO(TM_N)) &\cong (\ast_{N^2} C(S^1) )\rtimes_{\beta_1} \mathbb{Z}_N \rtimes_{\beta_2} \mathbb{Z}_N \\
M_N(\ast_{N^2}C(S^1)) &\cong \FO(TM_N)) \rtimes_{\alpha_1} \mathbb{Z}_N \rtimes_{\alpha_2} \mathbb{Z}_N  
\end{align*}
for all $ N \in \mathbb{N} $. 
\end{theorem}

In order to prove Theorem \ref{crossedproducts} we will construct the required isomorphisms explicitly, using again uniform notation 
to treat the cases of trivial and complete quantum graphs simultaneously. 

Consider the unitary error basis $ \W = \{X^j Z^k\}_{0 \leq j,k \leq N - 1} $ for $ M_N(\mathbb{C}) $ described in Remark \ref{pauli}.  
Moreover let $ \pi_N: \FO(\G_1(N)) \to \FO(\G_2(N)) \otimes M_N(\mathbb{C}) $ and $ \sigma_N: \FO(\G_2(N)) \to \FO(\G_1(N)) \otimes M_N(\mathbb{C}) $ be the 
corresponding embeddings constructed in the proof of Proposition \ref{Cuntzembedding}. That is, if we set 
$$ 
V^{rs}_{lm} = \frac{1}{N}(X^lZ^m)^*e_{rs}(X^lZ^m) = \frac{1}{N}\omega^{-(r-s)l}e_{r-m,s-m}, 
$$
and use our previous notation for the generators of $ \FO(\G_j(N)) $, then we obtain
$$
\pi_N(S^{(1)}_{lm}) = N \sum_{0 \le r,s \le N-1} S^{(2)}_{rs} \otimes V^{rs}_{lm} 
= \sum_{0 \le r,s \le N-1} S^{(2)}_{rs} \otimes \omega^{-(r-s)l} e_{r-m,s-m}
$$
and 
$$
\sigma_N(S^{(2)}_{jk}) = \sum_{0 \le l,m \le N-1} S^{(1)}_{lm} \otimes (V^{kj}_{lm})^{t}  
= \frac{1}{N} \sum_{0 \le l,m \le N-1} S^{(1)}_{lm} \otimes \omega^{(j-k)l} e_{j-m,k-m}.
$$
From the above formulas we can easily see that 
\begin{align*}
(1 \otimes X)\sigma_N(S^{(2)}_{jk})(1 \otimes X^*) &= \sigma_N(\alpha_1(S^{(2)}_{jk})) \\
(1 \otimes X)\pi_N(S^{(1)}_{jk})(1 \otimes X^*) &= \pi_N(\beta_1(S^{(1)}_{jk})). 
\end{align*} 
Hence $ (\sigma_N, (1 \otimes X)) $ defines a covariant representation of the dynamical system $ (\FO(\G_2(N)), \mathbb{Z}_N, \alpha_1) $, 
and similarly  $ (\pi_N, 1 \otimes X) $ defines a covariant representation of $ (\FO(\G_1(N)), \mathbb{Z}_N, \beta_1) $. 
As a consequence, we obtain $ \ast $-homomorphisms 
\begin{align*} 
\sigma_N'&: \FO(\G_2(N)) \rtimes_{\alpha_1} \mathbb{Z}_N \to \FO(\G_1(N)) \otimes M_N(\mathbb{C}) \\
\pi_N'&: \FO(\G_1(N)) \rtimes_{\beta_1} \mathbb{Z}_N \to \FO(\G_2(N)) \otimes M_N(\mathbb{C}) 
\end{align*} 
satisfying
\begin{align*}
\sigma_N'(a) &= \sigma_N(a), \qquad \sigma_N'(v) = 1 \otimes X, \\
\pi_N'(b) &= \pi_N(b), \qquad \pi_N'(v) = 1 \otimes X,  
\end{align*}
where $ a \in \FO(\G_2(N)) $ and $ b \in \FO(\G_1(N)) $, respectively. 
Similarly, we compute 
\begin{align*}
(1 \otimes Z^*) \sigma_N(S^{(2)}_{jk})(1 \otimes Z) &= \sigma_N(\alpha_2(S^{(2)}_{jk})), \\
(1 \otimes Z^*)(1 \otimes X) &= \omega (1 \otimes X)(1 \otimes Z^*), \\
(1 \otimes Z^*) \pi_N(S^{(1)}_{jk})(1 \otimes Z) &= \pi_N(\beta_2(S^{(1)}_{jk})). 
\end{align*} 
Hence $ (\sigma_N', (1 \otimes Z^*)) $ defines a covariant representation of the dynamical 
system $ (\FO(\G_2(N)) \rtimes_{\alpha_1} \mathbb{Z}_N, \mathbb{Z}_N, \alpha_2) $, 
and $ (\pi_N', (1 \otimes Z^*)) $ defines a covariant representation of $ (\FO(\G_1(N)) \rtimes_{\beta_1} \mathbb{Z}_N, \mathbb{Z}_N, \beta_2) $. 
In the same way as before we obtain associated $ \ast $-homomorphisms 
\begin{align*} 
\sigma_N''&: \FO(\G_2(N)) \rtimes_{\alpha_1} \mathbb{Z}_N \rtimes_{\alpha_2} \mathbb{Z}_N \to \FO(\G_1(N)) \otimes M_N(\mathbb{C}) \\
\pi_N''&: \FO(\G_1(N)) \rtimes_{\beta_1} \mathbb{Z}_N \rtimes_{\beta_2} \mathbb{Z}_N \to \FO(\G_2(N)) \otimes M_N(\mathbb{C}), 
\end{align*}
satisfying
\begin{align*}
\sigma_N''(a) &= \sigma_N(a), \qquad \sigma_N''(v) = \sigma_N'(v) = 1 \otimes X, \qquad \sigma_N''(w) = 1 \otimes Z^*,  \\
\pi_N''(b) &= \pi_N(b), \qquad \pi_N''(v) = \pi_N'(v) = 1 \otimes X, \qquad \pi_N''(w) = 1 \otimes Z^*,  
\end{align*}
respectively, where $ a \in \FO(\G_2(N)), b \in \FO(\G_1(N)) $. 

With these constructions in place, Theorem \ref{crossedproducts} is a consequence of the following assertion.   

\begin{theorem}
The maps $ \sigma_N'' $ and $ \pi_N'' $ are isomorphisms. 
\end{theorem} 

\begin{proof} 
Using $ C^\ast(X, Z^*) = M_N(\mathbb{C}) $ and the description of $ \sigma_N'' $ given above it is easy to see that $ \sigma_N'' $ is surjective. 
Explicitly, the range of $ \sigma_N'' $ contains $ \sigma_N(\FO(\G_2(N))) $ and $ 1 \otimes M_N(\mathbb{C}) $, 
and these two algebras generate $ \FO(\G_1(N)) \otimes M_N(\mathbb{C}) $. 

To show that $ \sigma_N'' $ is injective consider the $ \ast $-homomorphism 
$$
(\pi_N \otimes \id) \sigma_N: \FO(\G_2(N)) \to \FO(\G_2(N)) \otimes M_N(\mathbb{C}) \otimes M_N(\mathbb{C}). 
$$
If we denote by $ \{|\xi_{jk} \ket \}_{0 \le j,k \le N-1} \subset \mathbb{C}^N \otimes \mathbb{C}^N $ the orthonormal basis of maximally entangled 
vectors given by 
$$
\xi_{jk} = \frac{1}{\sqrt{N}} \sum_{r = 0}^{N-1} X^j Z^k |r \ket \otimes | r\ket, 
$$ 
then one obtains 
\begin{align*}
(\pi_N \otimes \id) \sigma_N''(S^{(2)}_{jk}) &= \frac{1}{N} \sum_{lmrs} S^{(2)}_{rs} \otimes \omega^{-(r-s)l} e_{r-m, s-m} \otimes \omega^{(j-k)l}e_{j-m,k-m} \\
&= \sum_{sm} S^{(2)}_{s+j-k,s} \otimes e_{s+j-k-m, s-m} \otimes e_{j-m, k-m} \\
&= \sum_{nm} S^{(2)}_{n+j, n+k} \otimes e_{n+j-m,n+k-m} \otimes e_{j-m,k-m} \\
&= \sum_{ln} \alpha_1^{-l}\alpha_2^{-n}(S^{(2)}_{jk}) \otimes |\xi_{l,n}\ket \bra \xi_{l,n}|. 
\end{align*}
Next, we define a unitary $ V $ on $ \mathbb{C}^N \otimes \mathbb{C}^N $ by setting $ V(|j\ket \otimes |k\ket) = \omega^{-jk}|\xi_{jk} \ket $. 
Then we have 
$$
V^*(1 \otimes X)V = Z \otimes 1, 
\qquad V^*(1 \otimes Z^*) V = X \otimes Z.
$$ 
Thus, if we consider the $ \ast $-homomorphism 
$$
\Phi: \FO(\G_2(N)) \rtimes_{\alpha_1} \mathbb{Z}_N \rtimes_{\alpha_2} \mathbb{Z}_N \to \FO(\G_2(N)) \otimes M_N(\mathbb{C}) \otimes M_N(\mathbb{C})
$$ 
given by $ \Phi = \ad(1 \otimes V^*)(\pi_N \otimes \id)\sigma_N'' $, then we get 
$$
\Phi(a) = \sum_{l,n} \alpha_1^{-l}\alpha_2^{-n}(a) \otimes e_{ll} \otimes e_{nn} 
$$
for all $ a \in \FO(\G_2(N)) $, and also 
$$
\Phi(v) = (1 \otimes V^*)(\pi_N \otimes \id) \sigma_N''(v)(1 \otimes V) = 1 \otimes V^*(1 \otimes X)V = 1 \otimes Z \otimes 1
$$ 
and 
$$
\Phi(w) = (1 \otimes V^*)(\pi_N \otimes \id) \sigma_N''(w)(1 \otimes V) = 1 \otimes V^*(1 \otimes Z^*)V = 1 \otimes X \otimes Z.
$$
From these formulas it follows that the image of $ \Phi $ is exactly the reduced crossed 
product $ \FO(\G_2(N)) \rtimes_{\alpha_1,r} \mathbb{Z}_N \rtimes_{\alpha_2,r} \mathbb{Z}_N $, 
and $ \Phi $ is none other than the canonical quotient map from the full crossed product to the reduced crossed product. 
Since $ \mathbb{Z}_N $ is finite, and hence amenable, the map $ \Phi $ is an isomorphism, forcing $ \sigma_N''$ to be injective. This proves the claim 
for $ \sigma_N'' $. 

For $ \pi_N'' $ one proceeds in a similar way, essentially by swapping the roles of the maps $ \pi_N $ and $ \sigma_N $ and repeating the above arguments.
\end{proof}

\begin{remark}
The first pair of isomorphisms in Theorem \ref{crossedproducts} 
should not come as a great surprise, given that Theorem \ref{quantumcompletemain} in section \ref{secexamples} already asserts an 
isomorphism $ \FO(K(M_N(\mathbb{C}), \tr)) \cong \O_{N^2} $. 
In fact, the latter isomorphism can be verified by considering the injective $ \ast $-homomorphism $ \sigma_N^K: \FO(K(M_N(\mathbb{C}), \tr)) \to M_N(\O_{N^2}) $ 
obtained in Proposition \ref{Cuntzembedding} and inspecting the relations in the proof of Proposition \ref{constructpiW}. In the next section we will 
prove Theorem \ref{quantumcompletemain} in full generality. 
\end{remark} 

\begin{remark}
Taking into account the identification $ \FO(K(M_N(\mathbb{C}), \tr)) \cong \O_{N^2} $, the statement for complete quantum graphs in Theorem \ref{crossedproducts} 
is reminiscent of Takesaki-Takai duality. However, the isomorphisms are slightly different. 
Note also that the $ C^\ast $-algebras $ *_{N^2} C(S^1) $ and $ \FO(TM_N) $ are not even Morita equivalent, compare Theorem \ref{Ktheoryquantumtrivial}.  
\end{remark} 

\begin{remark}
Using the isomorphism from Theorem \ref{Ktheoryquantumtrivial} we see that $ \pi_N^T $ induces an 
embedding $ \ast_{N^2} C(S^1) \rightarrow M_N(\mathbb{C}) *_1 (C(S^1) \oplus \mathbb{C}) $. In the notation used above this maps the 
generators $ S^{(1)}_{kl} $ to $ \sum_{rs} \omega^{k(s - r)} e_{r - l,r} S e_{s, s - l} $, where $ S $ denotes the standard generator 
of $ C(S^1) \subset C(S^1) \oplus \mathbb{C} $.  
\end{remark}

\begin{remark} 
It seems natural to look at pairs of quantum Cuntz-Krieger algebras associated to quantum isomorphic quantum graphs  
beyond the cases considered in Theorem \ref{crossedproducts}. Finding ``small'' representations of linking algebras could potentially allow one 
to transfer properties like unitality, nuclearity, or existence of traces from one algebra to the other, without a priori knowing 
whether the algebras are isomorphic or not. 
\end{remark}

\section{The structure of complete quantum Cuntz-Krieger algebras} \label{secquantumcomplete}

In this final section we discuss our main result on the structure of complete quantum Cuntz-Krieger algebras, that is, we 
provide the proof of Theorem \ref{quantumcompletemain} stated in section \ref{secexamples}. 

Let us begin with a simple lemma.

\begin{lemma} \label{cuntz-sym}
Let $ A $ be a non-zero unital $ C^\ast $-algebra and let $ n_1, n \in \mathbb{N} $. Moreover assume that $ u = (u_{xy}) \in M_{n_1, n}(A) $ is a 
rectangular unitary matrix with coefficients in $ A $. Let $ s_x $ for $ 1 \leq x \leq n_1 $ be the standard generators of $ \O_{n_1} $ and 
define elements $ \hat{s}_y \in \O_{n_1} \otimes A $ for $ 1 \leq y \leq n $ by 
$$
\hat{s}_y = \sum_{x = 1}^{n_1} s_x \otimes u_{xy}.
$$ 
Then the elements $ \hat{s}_y $ satisfy the defining relations of $ \O_n $ and 
$$
C^\ast(\hat{s}_1, \ldots, \hat{s}_n) \cong \O_n.
$$
\end{lemma}

\begin{proof}
In order to verify the Cuntz relations we calculate
\begin{align*}
\hat s_z^*\hat s_y &= \sum_{x_1,x_2} s_{x_1}^*s_{x_2} \otimes u_{x_1z}^*u_{x_2y} \\
&= \sum_{x_1} 1 \otimes u_{x_1z}^*u_{x_1y} \\
&= \delta_{y,z}(1 \otimes 1)
\end{align*}
and 
\begin{align*}
\sum_y \hat{s}_y \hat{s}_y^* &= \sum_{y, x_1, x_2} s_{x_1} s_{x_2}^* \otimes u_{x_1y} u_{x_2y}^* \\
&= \sum_{x_1, x_2} s_{x_1} s_{x_2}^* \otimes \delta_{x_1,x_2}1 \\
&= 1 \otimes 1.   
\end{align*}
Since $ \O_n $ is simple this yields the claim. 	
\end{proof}

Now let us fix a complete quantum graph $ K(B,\psi) $ satisfying the hypotheses of Theorem \ref{quantumcompletemain}, 
that is, $ (B, \psi) $ is a finite quantum space in standard form such that $ \psi: B \rightarrow \mathbb{C} $ is a $ \delta $-form 
with $ \delta^2 \in \mathbb{N} $. We shall use the same notation that as after Definition \ref{deffqs}, so 
that $ B = \bigoplus_{a = 1}^d M_{N_a}(\mathbb{C}) $ and $ \psi(x) = \sum_{a = 1}^d \Tr(Q_{(a)} x_i) $ for $ x = (x_1, \dots, x_d) $. 

By Remark \ref{wealth}, we have a quantum isomorphism $ K(B,\psi) \cong_q K_{\delta^2} $. 
Denote by $ v^x_{ija} $ for $ 1 \leq a \leq d, 1 \leq i,j \leq N_a, 1 \leq x \leq \delta^2 $ the standard generators of 
the $ C^\ast $-algebra $ A = C(G^+(K_{\delta^2}, K(B,\psi))) $ given in Proposition \ref{qisorelations}.  
Moreover let $ n = \dim(B) $ and consider the rectangular matrix $ u = (u_{ija}^x) \in M_{\delta^2,n}(A) $ given by 
$$
u_{ija}^x = (Q_{(a)})^{-1/2}_{jj} \delta^{-1} v_{ija}^x.
$$ 
Using the relations in Proposition \ref{qisorelations} one obtains 
\begin{align*} 
(u^*u)_{ija, klb} &= \sum_x (Q_{(a)})^{-1/2}_{jj} \delta^{-1} (v_{ija}^x)^* (Q_{(b)})^{-1/2}_{ll} \delta^{-1} v_{klb}^x \\
&= \sum_x (Q_{(a)})^{-1/2}_{jj} (Q_{(b)})^{-1/2}_{ll} \delta^{-2} v_{jia}^x v_{klb}^x \\
&= \sum_x (Q_{(a)})^{-1/2}_{jj} (Q_{(a)})^{-1/2}_{ll} \delta^{-2} \delta_{ab} \delta_{ik} v_{jla}^x \\
&= \sum_x (Q_{(a)})^{-1/2}_{jj} (Q_{(a)})^{-1/2}_{ll} (Q_{(a)})_{jj} \delta_{ab} \delta_{ik} \delta_{jl} \\
&= \delta_{ab} \delta_{ik} \delta_{jl} 
\end{align*} 
and 
\begin{align*}
(uu^*)_{xy} &= \sum_{ija} (Q_{(a)})^{-1/2}_{jj}\delta^{-1} v_{ija}^x (Q_{(a)})^{-1/2}_{jj}\delta^{-1} (v_{ija}^y)^* \\
&= \sum_{ija} (Q_{(a)})^{-1}_{jj} \delta^{-2} v_{ija}^x v_{jia}^y \\
&= \sum_{ia} \delta_{xy} v_{iia}^x \\
&= \delta_{xy}.
\end{align*} 
We conclude that $ u^*u = 1_{M_{n}(A)} $ and $ uu^* = 1_{M_{\delta^2}(A)} $, or equivalently, that $ u $ is unitary.  

Next, we consider the $ \ast $-homomorphism $ \hat{\beta}: \FO(K(B,\psi)) \rightarrow \O_{\delta^2} \otimes A $ from Theorem \ref{qisolift}, 
which satisfies 
$$
\hat{\beta}(S(e_{ij}^{(a)})) = \sum_{x} S(e_x) \otimes  v_{jia}^x
$$
in terms of the standard matrix units. Equivalently, if we write $ S_{ij}^{(a)} = S(f_{ij}^{(a)}) $, 
where $ f_{ij}^{(a)} = (Q_{(a)})^{-1/2}_{ii} e^{(a)}_{ij} (Q_{(a)})_{jj}^{-1/2}$ are the adapted matrix units for $ (B,\psi) $, and $ s_x = S(\delta^2e_x) $ 
for the canonical Cuntz isometries generating $ \O_{\delta^2} $, then 
\begin{align*}
\hat{\beta}(S_{ij}^{(a)}) &=(Q_{(a)})^{-1/2}_{ii}(Q_{(a)})^{-1/2}_{jj}\delta^{-2}\sum_{x} s_x \otimes v_{ija}^x\\
&= (Q_{(a)})_{ii}^{-1/2} \delta^{-1} \sum_{x} s_x \otimes u^x_{ija}.
\end{align*}
Hence the unitarity of the matrix $ u = (u_{ija}^x) $ combined with Lemma \ref{cuntz-sym} implies that the 
elements $ (Q_{(a)})^{1/2}_{ii} \delta \hat{\beta}(S_{ij}^{(a)}) $ form an $ n $-tuple of Cuntz isometries in $ \O_{\delta^2} \otimes A $. 
According to Remark \ref{remhatinjective}, the restriction of $ \hat{\beta} $ to the $ \ast $-algebra generated by the $ S_{ij}^{(a)} $ is injective.  
This shows that $ \FO(K(B,\psi)) $ is unital with unit 
$$ 
e = \sum_{ija} (Q_{(a)})_{ii} \delta^2 S_{ij}^{(a)}(S_{ij}^{(a)})^*,  
$$  
and that the elements $ (Q_{(a)})^{1/2}_{ii} \delta S_{ij}^{(a)} $ form an $ n $-tuple of Cuntz isometries generating $ \FO(K(B,\psi)) $. 
This completes the proof of Theorem \ref{quantumcompletemain}.

\begin{remark}
It seems reasonable to expect that $ \FO(K(B,\psi)) \cong \O_n $ for all choices of $ \delta $-forms $ \psi $, but we are unable to supply a proof. 
Note that when $ \delta^2 \notin \mathbb{N} $, we no longer have a quantum isomorphism between $ K(B,\psi) $ and a classical complete graph, and 
therefore a different approach would be needed. 
\end{remark}

\bibliographystyle{hacm}

\bibliography{cvoigt}

\end{document}